%% file: main.tex
\documentclass[onefignum,onetabnum]{siamart171218}

\usepackage{lipsum}
\usepackage{amsfonts}
\usepackage{graphicx}
\usepackage{epstopdf}
\usepackage{algorithm}
\usepackage{algorithmic}
\usepackage{tikz}
\usetikzlibrary{arrows}

\usepackage{multirow}
\usepackage{listings}
\usepackage{mathtools}
\usepackage{latexsym,amsmath,amsfonts,amscd}
\usepackage{bbm}
\usepackage{caption}
\usepackage{subcaption}

\newtheorem{thm}{Theorem}
\newtheorem{rmk}{Remark}

\DeclareMathOperator*{\vecm}{vec}

\ifpdf
  \DeclareGraphicsExtensions{.eps,.pdf,.png,.jpg}
\else
  \DeclareGraphicsExtensions{.eps}
\fi

\headers{Compact FD for Incompressible Flow}{H. Li and X. Zhang}

\title{A high order accurate bound-preserving compact finite difference scheme for two-dimensional incompressible flow
}

\author{Hao Li
\and Xiangxiong Zhang \thanks{Department of Mathematics,
Purdue University,
150 N. University Street,
West Lafayette, IN 47907-2067 (\email{li2497@purdue.edu}, \email{zhan1966@purdue.edu}). Research is supported by NSF DMS-1913120.}}

\usepackage{amsopn}
\graphicspath{{graphics/}}

\begin{document}

\maketitle

\begin{abstract}
For solving two-dimensional incompressible flow in the vorticity form by the fourth-order compact finite difference scheme and explicit strong stability preserving (SSP) temporal discretizations, we show that the simple bound-preserving limiter in \cite{li2018high} can enforce the strict bounds of the vorticity, if the velocity field satisfies a discrete divergence free constraint. For reducing oscillations, a modified TVB limiter adapted from \cite{cockburn1994nonlinearly} is constructed without affecting the bound-preserving property. This bound-preserving finite difference method can be used for any passive convection equation with a divergence free velocity field.
\end{abstract}



\section{Introduction}
In this paper, we are interested in constructing high order compact finite difference schemes solving the following two-dimensional time-dependent incompressible Euler equation in vorticity and  stream-function formulation
\begin{subequations}
\label{incompressibleEuler}\begin{eqnarray}
\omega_t+(u\omega)_x+(v\omega)_y &=& 0, \label{incompressible1}\\
\psi&=&\Delta\omega, \label{poisson1}\\
 \langle u,v\rangle &=&\langle-\psi_y,\psi_x\rangle, \label{uv1}
\end{eqnarray} 
\end{subequations}
with periodic boundary conditions and suitable initial conditions. In the above formulation, $\omega$ is the vorticity, $\psi$ is the stream function, $\langle u,v\rangle $ is the velocity and $Re$ is the Reynolds number.

For simplicity, we only focus on the incompressible Euler equation \eqref{incompressibleEuler}. With explicit time discretizations, the extension of high order accurate bound-preserving compact finite difference scheme to Navier-Stokes equation
\begin{equation}
\label{incompressibleNS}
\omega_t+(u\omega)_x+(v\omega)_y = \frac{1}{Re}\Delta \omega
\end{equation}
would be straightforward following the approach in \cite{li2018high}. 

The equation \eqref{uv1} implies the incompressbilility condition 
\begin{equation}
u_x+v_y=0.
\label{divfree}
\end{equation}
Due to \eqref{divfree}, \eqref{incompressible1} is equivalent to
\begin{equation}
\omega_t+u\omega_x+v\omega_y =0 
\label{nonconserv-incom}
\end{equation}
for which the initial value problem satisfies a bound-preserving property:
\[\min_{x,y} \omega(x,y,0)=m\leq \omega(x,y,t)\leq M=\max_{x,y} \omega(x,y,0).\]

If solving \eqref{nonconserv-incom} directly, it is usually easier to construct a bound-preserving scheme. For the sake of conservation, 
it is desired to solve the conservative form equation \eqref{incompressible1}. The divergence free constraint \eqref{divfree} is one of the main difficulties in solving incompressible flows. In order to enforce the bound-preserving property for \eqref{incompressible1} without losing accuracy, the incompressibility condition must be properly used since the bound-preserving property may not hold for 
\eqref{incompressible1} without \eqref{divfree}, see \cite{zhang2010maximum, zhang2012maximum2, zhang2013maximum}.

Even though the bound-preserving property and the global conservation imply certain nonlinear stability, in practice a bound-preserving high order accurate compact finite difference scheme can still produce excessive oscillations for a pure convection problem. Thus an additional limiter for reducing oscillations is often needed, e.g., 
 the total variation bounded (TVB) limiter discussed in \cite{cockburn1994nonlinearly}.
One of the main focuses of this paper is to design suitable TVB type limiters, without losing bound-preserving property. Notice that the TVB limiter for a compact finite difference scheme is designed in a quite different way from those for discontinuous Galerkin method, thus it is nontrivial to have a bound-preserving TVB limiter for the compact finite difference schemes.


The paper is organized as follows.  Section \ref{sec-cfd-review} is a review of the compact finite difference method and a simple bound-preserving limiter for scalar convection-diffusion equations.  In Section \ref{sec-incompressible1}, we show that the compact finite difference scheme can be rendered bound-preserving if the velocity field satisfies a discrete divergence free condition.  We discuss the bound-preserving property of a TVB limiter in Section \ref{sec-tvb-limiter}.  Numerical tests are shown in Section \ref{sec-numeric-test}. Concluding remarks are given in Section \ref{sec-rmk}.
\section{Review of compact finite difference method}
\label{sec-cfd-review}
In this section we review the compact finite difference method and a bound-preserving limiter in \cite{li2018high}.  

\subsection{A fourth-order accurate compact finite difference scheme}
Consider a smooth  function $f(x)$ on the interval $[0,1]$.
Let $x_i=\frac{i}{N}$ $(i=1,\cdots, N)$ be the uniform grid points on the interval $[0,1]$.  A fourth-order accurate compact finite difference approximation to derivatives on the interval $[0,1]$ is given as:
\begin{equation}\label{4th-cfd}
\begin{aligned}
\frac16( f'_{i+1}+4f'_i +f'_{i-1})= & \frac{f_{i+1}-f_{i-1}}{2\Delta x}+\mathcal{O} (\Delta x^4),\\
\frac{1}{12}( f''_{i+1}+4f''_i +f''_{i-1})= & \frac{f_{i+1}-2f_{i-1} + f_{i-1}}{\Delta x^2}+\mathcal{O} (\Delta x^4),
\end{aligned}
\end{equation}
where $f_i$, $f'_i$and $f''_i$ are point values of a function $f(x)$,  its derivative $f'(x)$ and its second order derivative $f''(x)$ at uniform grid points $x_i$ $(i=1,\cdots,N)$ respectively.

Let $\mathbf{f}$ be a column vector with numbers $f_1, f_2, \cdots, f_N$ as entries.
Let $W_1$, $W_2$, $D_x$ and $D_{xx}$ denote four linear operators as follows:
\begin{equation}\label{fir-de-diff}
 W_1\mathbf{f}=\frac16\begin{pmatrix}
              4 & 1 & & & 1\\
              1 & 4 & 1 & &  \\              
              &  \ddots & \ddots & \ddots & \\
              &  & 1 & 4 &1 \\
              1  & & & 1 & 4 \\
             \end{pmatrix}
             \begin{pmatrix}
 f_1 \\
 f_2\\
 \vdots\\
 f_{N-1}\\
 f_N
 \end{pmatrix}, D_x\mathbf{f} =
 \frac{1}{2}\begin{pmatrix}
              0 & 1 & & &  -1\\
              -1 & 0 & 1 & &  \\              
              &  \ddots & \ddots & \ddots & \\
              &  & -1 & 0 & 1 \\
              1  & & & -1 & 0 \\
             \end{pmatrix} 
  \begin{pmatrix}
 f_1 \\
 f_2\\
 \vdots\\
 f_{N-1}\\
 f_N
 \end{pmatrix},
\end{equation}
\begin{equation}\label{sec-de-diff}
W_2\mathbf{f}=\frac{1}{12}\begin{pmatrix}
              10 & 1 & & & 1\\
              1 & 10 & 1 & &  \\              
              &  \ddots & \ddots & \ddots & \\
              &  & 1 & 10 &1 \\
              1  & & & 1 & 10 \\
             \end{pmatrix}
             \begin{pmatrix}
 f_1 \\
 f_2\\
 \vdots\\
 f_{N-1}\\
 f_N
 \end{pmatrix},
 D_{xx}\mathbf{f} =
 \begin{pmatrix}
              -2 & 1 & & &  1\\
              1 & -2 & 1 & &  \\              
              &  \ddots & \ddots & \ddots & \\
              &  & 1 & -2 & 1 \\
              1  & & & 1 & -2 \\
             \end{pmatrix}
  \begin{pmatrix}
 f_1 \\
 f_2\\
 \vdots\\
 f_{N-1}\\
 f_N
 \end{pmatrix}.
\end{equation}

If $f(x)$ is periodic with with period $1$, the fourth-order compact finite difference approximation \eqref{4th-cfd} to the first order derivative and second order derivative can be denoted as
\begin{align*}
W_1 \mathbf{f}'= \frac{1}{\Delta x}D_x \mathbf{f}, \quad W_2 \mathbf{f}''= \frac{1}{\Delta x^2} D_{xx} \mathbf{f},
\end{align*}
which can be explicitly written as 
\begin{align*}
\mathbf{f}'=  \frac{1}{\Delta x} W_1^{-1} D_x \mathbf{f},\quad \mathbf{f}''=  \frac{1}{\Delta x^2} W_2^{-1} D_{xx} \mathbf{f},
\end{align*}
where $W_1^{-1}$ and $W_2^{-1}$ are the inverse operators. 
For convenience, by abusing notations we let $W_1^{-1} f_i$ denote the $i$-th entry of the vector $W_1^{-1} \mathbf f$. 

\subsection{High order time discretizations}\label{sec-time-dis}
\label{sec-highordertime}
For time discretizations, we use the strong stability preserving (SSP) Runge-Kutta and multistep methods, which are convex combinations of formal forward Euler steps. Thus we only need to discuss the bound-preserving for one forward Euler step since convex combination can preserve the bounds. 

For the numerical tests in this paper, we use a third order explicit SSP Runge–Kutta method SSPRK(3,3), see \cite{gottlieb2011strong}, which is widely known as the Shu-Osher method, with SSP coefficient $C = 1$ and effective SSP coefficient $C_{eff}=\frac{1}{3}$. For solving $u_t = F (u)$, it is given by
\begin{equation*}
\begin{aligned}
&u^{(1)}=u^n, \\
&u^{(2)}=u^{(1)}+dt F(u^{(1)}), \\
&u^{(3)}=\frac{3}{4}u^{(1)} + \frac{1}{4}(u^{(2)}+F(u^{(2)})), \\
&u^{n+1} =\frac{1}{3} u^{(1)}+\frac{2}{3}(u^{(3)}+F(u^{(3)})).
\end{aligned}
\end{equation*}

\subsection{A three-point stencil bound-preserving limiter}\label{sec-3pt-bp-limiter}
In this subsection, we review the three-point stencil bound-preserving limiter in \cite{li2018high}. Given a sequence of periodic point values $u_i$ $(i=1,\cdots, N)$, $u_0:=u_N$, $u_{N+1}:=u_1$ and constant $a\geq 2$, assume all local weighted averages are  in the range $[m, M]$:
 \begin{equation*}
 m\leq \frac{1}{a+2}( u_{i-1}+ au_{i}+ u_{i+1})\leq M, \quad i=1,\cdots, N,\quad a\geq 2.
 \end{equation*}

We separate the point values $\{u_i, i=1,\cdots, N\}$ into two classes of subsets consisting of consecutive point values. 
 In the following discussion, a {\it set} refers to a set of consecutive point values $u_l, u_{l+1}, u_{l+2}, \cdots, u_{m-1}, u_m$.
For any set $S=\{u_l, u_{l+1}, \cdots, u_{m-1}, u_m\}$, we call the first point value $u_l$ and the last point value $u_m$ as {\it boundary points},
and call the other point values  ${u_{l+1}, \cdots, u_{m-1}}$ as  {\it interior points}.
A set of class I is defined as a set satisfying the following:
\begin{enumerate}
 \item It contains at least four point values.
 \item Both {\it boundary points} are in $[m,M]$ and all {\it interior points} are out of range. 
 \item It contains both undershoot and overshoot points.
\end{enumerate}
Notice that in a set of class I, at least one undershoot point is next to an overshoot point.
For given point values $u_i,i=1,\cdots, N$, 
suppose all the sets of class I are $S_1=\{u_{m_1},u_{m_1+1},\cdots, u_{n_1}\}$, $S_2=\{u_{m_2},\cdots, u_{n_2}\}$, $\cdots$, $S_K=\{u_{m_K},\cdots, u_{n_K}\}$, 
where $m_1<m_2<\cdots< u_{m_K}$.

A set of class II consists of point values between $S_i$ and $S_{i+1}$ and two boundary points $u_{n_i}$ and $u_{m_{i+1}}$. 
Namely they are
$T_{0}=\{u_1,u_2,\cdots,u_{m_1}\}$, $T_{1}=\{u_{n_1},\cdots,u_{m_2}\}$, $T_{2}=\{u_{n_2},\cdots,u_{m_3}\}$, $\cdots$, $T_{K}=\{u_{n_K},\cdots,u_N\}$.
For periodic data $u_i$, we can combine $T_{K}$ and $T_0$ to define $T_K=\{u_{n_K},\cdots,u_{N},u_1,\cdots,u_{m_1}\}$.

In the sets of class I, the undershoot and the overshoot are neighbors. 
In the sets of class II,  the undershoot and the overshoot are separated, i.e., an overshoot is not next to any undershoot.
As a matter of fact, in the numerical tests, the sets of class I are hardly encountered. Here we include them in the discussion for the sake of completeness.
When there are no sets of class I, all point values form a single set of class II.

\begin{algorithm}[H]
\caption{A bound-preserving limiter for periodic data $u_i$ satisfying $\bar u_i\in [m, M]$}
\begin{algorithmic}[1]
\REQUIRE the input $u_i$ satisfies $\bar u_i=\frac{1}{a+2}( u_{i-1}+a u_{i}+ u_{i+1})\in [m, M]$, $a\geq 2$. Let $u_0$, $u_{N+1}$ denote $u_N$, $u_1$ respectively. 
\ENSURE the output satisfies $v_i\in[m, M], i=1,\cdots, N$ and $\sum_{i=1}^N v_i=\sum_{i=1}^N u_i$.
\STATE {\bf Step 0}: First set $v_i=u_i$, $i=1,\cdots, N$. Let $v_0$, $v_{N+1}$ denote $v_N$, $v_1$ respectively. 
\STATE {\bf Step I}: Find all the sets of class I $S_1,\cdots,S_K$ (all local saw-tooth profiles) and all the sets of class II $T_{1},\cdots, T_{K}$.
\STATE {\bf Step II}: For each $T_j$ $(j=1,\cdots, K)$, 
\FOR{all index $i$ in $T_j$}
\IF{$u_i <m$}
\STATE $v_{i-1}\leftarrow  v_{i-1}-\frac{(u_{i-1}-m)_+}{(u_{i-1}-m)_+ +(u_{i+1}-m)_+}(m-u_i)_+$
\STATE $v_{i+1} \leftarrow v_{i+1}-\frac{(u_{i+1}-m)_+}{(u_{i-1}-m)_+ +(u_{i+1}-m)_+}(m-u_i)_+$
\STATE $v_i\leftarrow m$
\ENDIF
\IF{$u_i >M$}
\STATE $v_{i-1}\leftarrow  v_{i-1}+\frac{(M-u_{i-1})_+}{(M-u_{i-1})_+ +(M-u_{i+1})_+}(u_i-M)_+$
\STATE $v_{i+1} \leftarrow v_{i+1}+\frac{(M-u_{i+1})_+}{(M-u_{i-1})_+ +(M-u_{i+1})_+}(u_i-M)_+$ 
\STATE $v_i\leftarrow M$
\ENDIF
\ENDFOR
\STATE {\bf Step III}: for each saw-tooth profile $S_j=\{u_{m_j},\cdots, u_{n_j}\}$ $(j=1,\cdots, K)$,
 let $N_0$ and $N_1$ be the numbers of undershoot and overshoot points in $S_j$ respectively. 
 \STATE Set $U_j=\sum_{i=m_j}^{n_j }v_i$.
\FOR{$i=m_j+1,\cdots, n_j-1$}
\IF{$u_i >M$}
\STATE $v_i\leftarrow M$.
\ENDIF
\IF{$u_i <m$}
\STATE $v_i\leftarrow m$.
\ENDIF
\ENDFOR

\STATE Set $V_j=N_1M +N_0 m+v_{m_j}+v_{n_j}$.
\STATE Set $A_j=v_{m_j}+v_{n_j}+N_1M-(N_1+2)m$, $B_j=(N_0+2) M-v_{m_j}-v_{n_j}-N_0 m$.
\IF {$V_j-U_j>0$}
\FOR{$i=m_j,\cdots, n_j$}
\STATE $v_i\leftarrow v_i-\frac{v_i-m}{A_j}(V_j-U_j)$
\ENDFOR
\ELSE
\FOR{$i=m_j,\cdots, n_j$}
\STATE $v_i \leftarrow v_i+\frac{M-v_i}{B_j}(U_j-V_j)$
\ENDFOR
\ENDIF
\end{algorithmic}
\label{bp-limiter}
\end{algorithm}

The algorithm \ref{bp-limiter} can enforce $\bar u_i\in [m, M]$ without losing conservation \cite{li2018high}:
\begin{thm}\label{1dconvectionthm}
Assume periodic data $u_i (i=1,\cdots, N)$ satisfies
$\bar u_i=\frac{1}{a+2}( u_{i-1}+a u_{i}+ u_{i+1})\in [m, M]$ for some fixed $a\geq 2$ and all $i=1,\cdots, N$ with $u_0:=u_N$ and $u_{N+1}:=u_1$,
then the output  of Algorithm \ref{bp-limiter} 
satisfies $\sum\limits_{i=1}^N v_i=\sum\limits_{i=1}^N u_i$ and $v_i\in [m,M],$ $\forall i$. 
\end{thm}

For the two-dimensional case, the same limiter can be used in a dimension by dimension fashion to enforce $u_{ij} \in [m,M]$.  

\section{A bound-preserving scheme for the two-dimensional incompressible flow}
\label{sec-incompressible1}

In this section we first show the fourth-order compact finite difference  with forward Euler time discretization satisfies the weak monotonicity \cite{li2018high}, thus it is bound-preserving with a naturally constructed discrete divergence-free velocity field.

For simplicity, we only consider a periodic boundary condition on a square $[0,1]\times [0,1]$.  
Let $(x_{i},y_j)=(\frac{i}{N_x},\frac{j}{N_y})$ $(i=1,\cdots, N_x,j=1, \cdots, N_y)$ be the uniform grid points on the domain $[0,1]\times[0,1]$.  All notation in this paper is consistent with those in
\cite{li2018high}.

\subsection{Weak monotonicity and bound-preserving}
\label{sec-euler}
 
Let $\lambda_1=\frac{\Delta t}{\Delta x}$ and  $\lambda_2=\frac{\Delta t}{\Delta y}$, the fourth-order compact finite difference scheme with the forward Euler method for \eqref{incompressible1}  can be given as 
\begin{eqnarray}\label{incompscheme1}
\omega^{n+1}_{ij}& = &\omega^n_{ij}-\lambda_1 [W_{1x}^{-1}D_x(\mathbf u^n\circ\boldsymbol{\omega}^n)]_{ij}-\lambda_2 [W_{1y}^{-1}D_y(\mathbf u^n\circ\boldsymbol{\omega}^n)]_{ij}.
\end{eqnarray}

With the same notation as in \cite{li2018high}, the weighted average in two dimensions can be denoted as 
\begin{equation}\label{avg-weight1}
    \bar{\mathbf{\omega}}=W_{1 x} W_{1 y} \mathbf{\omega}.
\end{equation}
Then the scheme \eqref{incompscheme1} is equivalent to
\begin{align}
\label{incompscheme2} 
\notag \bar{\omega}^{n+1}_{ij}&=\bar{\omega}^n_{ij}-\lambda_1[W_{1y}D_x(\mathbf u^n\circ\boldsymbol{\omega}^n)]_{ij}-\lambda_2 [W_{1x}D_y(\mathbf v^n\circ\boldsymbol{\omega}^n)]_{ij}\\
  &=\frac{1}{36}\begin{pmatrix}
                       1 & 4 & 1\\
                       4 & 16 & 4\\
                       1 & 4 & 1 \end{pmatrix}
                       : \Omega^n
 - \frac{\lambda_1}{12}\begin{pmatrix}
                       -1 & 0 & 1\\
                       -4 & 0 & 4\\
                       -1 & 0 & 1 \end{pmatrix}
                       : (U^n\circ\Omega^n)
 - \frac{\lambda_2}{12}\begin{pmatrix}
                       1 & 4 & 1\\
                       0 & 0 & 0\\
                       -1 & -4 & -1 \end{pmatrix}
                       : (V^n\circ\Omega^n),    
\end{align}
where $\circ$ denotes  the matrix Hadamard product, and
\[U=\begin{pmatrix*}[l]
  u_{i-1, j+1} &u_{i, j+1} &u_{i+1, j+1} \\
  u_{i-1, j} &u_{i, j} &u_{i+1, j} \\
  u_{i-1, j-1} &u_{i, j-1} &u_{i+1, j-1}
 \end{pmatrix*}, V=\begin{pmatrix*}[l]
  v_{i-1, j+1} &v_{i, j+1} &v_{i+1, j+1} \\
  v_{i-1, j} &v_{i, j} &v_{i+1, j} \\
  v_{i-1, j-1} &v_{i, j-1} &v_{i+1, j-1}
 \end{pmatrix*},\]
 \[\Omega=\begin{pmatrix*}[l]
  \omega_{i-1, j+1} &\omega_{i, j+1} &\omega_{i+1, j+1} \\
  \omega_{i-1, j} &\omega_{i, j} &\omega_{i+1, j} \\
  \omega_{i-1, j-1} &\omega_{i, j-1} &\omega_{i+1, j-1}
 \end{pmatrix*}.\]
It is straightforward to verify the {\it weak monotonicity}, i.e.,
$\bar{\omega}^{n+1}_{ij}$ is a monotonically increasing function with respect to all point values $\omega^n_{ij}$ involved in \eqref{incompscheme2}
under the CFL condition 
\[\frac{\Delta t}{\Delta x}\max_{ij} |u^n_{ij}|+\frac{\Delta t}{\Delta y}\max_{ij} |v^n_{ij}|\leq \frac13.\]
However,  the monotonicity is sufficient for bound-preserving $\bar \omega^{n+1}_{ij}\in[m, M]$, only if the following consistency condition holds:
\begin{equation}
\omega^n_{ij}\equiv m\Rightarrow \bar \omega^{n+1}_{ij}=m, \quad \omega^n_{ij}\equiv M\Rightarrow \bar \omega^{n+1}_{ij}=M.
\label{consistency}
\end{equation}
Plugging $\omega^n_{ij}\equiv m$ in \eqref{incompscheme2}, we get 
\begin{eqnarray*}
\begin{aligned}
\bar{\omega}^{n+1}_{ij} =&m\left(1-\lambda_1 \left(W_{1y}D_x\mathbf u^n\right)_{ij}-\lambda_2 \left(W_{1x}D_y\mathbf v^n\right)_{ij}\right).
\end{aligned}
\end{eqnarray*}
Thus the consistency \eqref{consistency} holds only if the velocity $\langle \mathbf u^n, \mathbf v^n\rangle$ satisfies:
\begin{equation}
\label{numdivefree}
\frac{1}{\Delta x}D_xW_{1y}\mathbf{u}^n+ \frac{1}{\Delta x}D_yW_{1x}\mathbf{v}^n = 0.
\end{equation}

Therefore we have the following bound-preserving result:
\begin{thm}
If the velocity $\langle \mathbf u^n, \mathbf v^n\rangle$ satisfies the discrete divergence free constraint \eqref{numdivefree} and $\omega^n_{ij}\in[m, M]$, 
then under the CFL constraint \[\frac{\Delta t}{\Delta x}\max_{ij} |u^n_{ij}|+\frac{\Delta t}{\Delta y}\max_{ij} |v^n_{ij}|\leq \frac13,\]
the scheme \eqref{incompscheme2} satisfies $\bar \omega^{n+1}_{ij}\in[m, M]$.
\end{thm}

\subsection{A discrete divergence free velocity field}
In the following discussion, we may discard the superscript $n$ for convenience assuming everything discussed is at time step $n$.

Note that \eqref{numdivefree} is a discrete divergence free constraint and we can construct a fourth-order accurate velocity field satisfying  \eqref{numdivefree}.  Given $\omega_{ij}$, we first compute $\psi_{ij}$ by a fourth-order compact finite difference scheme for the Poisson equation  \eqref{poisson1}. The detail of the Poisson solvers including the fast Poisson solver is given in the appendices.  

With the fourth-order compact finite difference we have
\begin{eqnarray}\label{reconstruct-uv}
- \frac{1}{\Delta y}D_y\mathbf{\Psi} = W_{1y}\mathbf{u},\quad  \frac{1}{\Delta x}D_x\mathbf{\Psi} = W_{1x} \mathbf{v},
\end{eqnarray}
where $$\mathbf{\Psi}=\begin{pmatrix}
              \psi_{11} & \psi_{12} & \cdots & \psi_{1, N_y}\\
              \psi_{21} & \psi_{22} & \cdots & \psi_{2, N_y} \\              
              \vdots &  \vdots & \ddots  & \vdots\\
              \psi_{N_x-1,1}& \psi_{N_x-1,2} &\cdots  & \psi_{N_x-1,N_y} \\
              \psi_{N_x,1}  & \psi_{N_x,2}&\cdots & \psi_{N_x,N_y}  \\
             \end{pmatrix}_{N_x\times N_y}.$$
Since the two finite difference operators $D_x$ and $D_y$ commute, it is straightforward to verify that the velocity field computed by \eqref{reconstruct-uv} satisfies \eqref{numdivefree}. 
\subsection{A fourth-order accurate bound-preserving scheme}

For the Euler equations \eqref{incompressibleEuler}, 
the following implementation of the fourth-order compact finite difference with forward Euler time discretization scheme can preserve the bounds:
\begin{enumerate}
 \item Given $\omega^n_{ij}\in [m, M]$, 
 solve the Poisson equation \eqref{poisson1} by the fourth-order accurate compact finite difference scheme to obtain point values of the stream function $\psi_{ij}$. 
 \item Construct $\mathbf{u}$ and $\mathbf{v}$ by \eqref{reconstruct-uv}.
 \item Obtain $\bar \omega^{n+1}_{ij}\in [m, M]$ by scheme \eqref{incompscheme2}. 
 \item Apply the limiting procedure in Section \ref{sec-3pt-bp-limiter} to obtain $\omega^{n+1}_{ij}\in [m, M]$.
\end{enumerate}
For high order SSP time discretizations, we should use the same implementation above for each time stage or time step. 

For the Navier-Stokes equations \eqref{incompressibleNS}, with $\mu_1=\frac{\Delta t}{\Delta x^2}$ and  $\mu_2=\frac{\Delta t}{\Delta y^2}$, the scheme can be written as
\begin{equation}\label{incompNSscheme1}
\begin{aligned}
\omega^{n+1}_{ij} = &\omega^n_{ij}-\lambda_1 [W_{1x}^{-1}D_x(\mathbf u^n\circ\boldsymbol{\omega}^n)]_{ij}
-\lambda_2 [W_{1y}^{-1}D_y(\mathbf v^n\circ\boldsymbol{\omega}^n)]_{ij}\\
&+\frac{\mu_1}{Re} W_{2x}^{-1}D_{xx} \omega^n_{ij}+\frac{\mu_2}{Re} W_{2y}^{-1}D_{yy} \omega^n_{ij},
\end{aligned}
\end{equation}

In a manner similar to \eqref{avg-weight1}, we define
\begin{equation}\label{avg-weight2}
    \tilde{\mathbf{\omega}}:=W_{2 x} W_{2 y} \mathbf{\omega},
\end{equation}
with $W_{1}:=W_{1 x} W_{1 y}$ and $W_{2}:=$ $W_{2 x} W_{2 y}$. Due to definition \eqref{avg-weight1} and the fact operators $W_{1}$ and $W_{2}$ commute, i.e. $W_{1} W_{2}=W_{2} W_{1}$, we have
$$
\tilde{\bar{\mathbf{\omega}}}=W_{2}W_{1}\mathbf{\omega}=W_{1}W_{2} \mathbf{\omega}=\bar{\tilde{\mathbf{\omega}}}.
$$

Then scheme \eqref{incompNSscheme1} is equivalent to 
\begin{equation}\label{scheme-ns}
\begin{aligned}
\tilde{\bar\omega}^{n+1}_{ij} = & \tilde{\bar\omega}^n_{ij}-\frac{\lambda_1}{12}[W_{2}W_{1y}Dx(\mathbf u^n\circ\boldsymbol{\omega}^n)]_{ij}
 - \frac{\lambda_2}{12}[W_{2}W_{1x}Dy(\mathbf u^n\circ\boldsymbol{\omega}^n)]_{ij} \\
+ & \frac{\mu_1}{Re} W_{1}W_{2y}D_{xx} \omega^n_{ij}+\frac{\mu_2}{Re} W_{1}W_{2x}D_{yy} \omega^n_{ij}.
\end{aligned}
\end{equation}

Following the discussion in Section \ref{sec-euler} and the discussion for the two-dimensional convection-diffusion in \cite{li2018high}, 
we have the following result:
\begin{thm}
If the velocity $\langle \mathbf u^n,\mathbf v^n\rangle$ satisfies the constraint \eqref{numdivefree} and $\omega^n_{ij}\in[m, M]$, 
then under the CFL constraint 
$$\frac{\Delta t}{\Delta x}\max_{ij} |u^n_{ij}|+\frac{\Delta t}{\Delta y}\max_{ij} |v^n_{ij}|\leq \frac16,\quad 
\frac{\Delta t}{Re \Delta x^2}+\frac{\Delta t}{Re \Delta y^2} \leq \frac{5}{24},$$
the scheme \eqref{scheme-ns} satisfies  $\tilde{\bar\omega}^{n+1}_{ij}\in[m, M]$. 
\end{thm}

Given $\tilde{\bar \omega}_{ij}$, we can recover point values $\omega_{ij}$ by obtaining first $\tilde \omega_{ij}=W_{1}^{-1}\tilde{\bar \omega}_{ij}$
then $\omega_{ij}=W_{2}^{-1} \tilde \omega_{ij}$.
Given point values $\omega_{ij}$ satisfying $\tilde{\bar \omega}_{ij} \in[m,M]$ for any $i$ and $j$, 
we can use the limiter in Algorithm \ref{bp-limiter} in a dimension by dimension fashion several times to enforce $\omega_{ij}\in[m, M]$:
\begin{enumerate}
 \item Given  $\tilde{\bar \omega}_{ij}\in [m, M]$, compute $\tilde \omega_{ij}=W_{1}^{-1}\tilde{\bar \omega}_{ij}$ and
 apply the limiting Algorithm \ref{bp-limiter} with $a=4$ in both $x$-direction and $y$-direction to ensure  $\tilde \omega_{ij}\in [m, M]$.
 \item  Given  $\bar{\omega}_{ij}\in [m, M]$, compute $\omega_{ij}=W_{2}^{-1}\tilde{\omega}_{ij}$ and
 apply the limiting algorithm Algorithm \ref{bp-limiter} with $a=10$ in both $x$-direction and $y$-direction to ensure $\omega_{ij}\in [m, M]$.
 \end{enumerate}

\section{A TVB limiter for the two-dimensional incompressible flow}
\label{sec-tvb-limiter}
To have nonlinear stability and eliminate oscillations for shocks, a TVBM (total variation bounded in the means) limiter was introduced
for the compact finite difference scheme solving scalar convection equations in \cite{cockburn1994nonlinearly}.
In this section, we will modify this limiter  for the incompressible flow so that it does not affect the bound-preserving property.
 Thus we can use both the TVB limiter and the bound-preserving limiter in Algorithm \ref{bp-limiter} to preserve bounds while reducing oscillations. For simplicity, we only consider the numerical scheme for the incompressible Euler equations \eqref{incompressibleEuler}. In this section, we may discard the superscript $n$ if a variable is defined at time step $n$.

\subsection{The TVB limiter}\label{sec-tvb-limiter-def}
The scheme \eqref{incompscheme2}  can be written in a conservative form:
\begin{equation}\label{incomp-conservationform}
 \bar{\omega}_{ij}^{n+1}=\bar{\omega}_{ij}^n-\lambda_1[\hat{(u\omega)}^n_{i+\frac{1}{2},j}-\hat{(u\omega)}^n_{i-\frac{1}{2},j}]-\lambda_2[\hat{(v\omega)}^n_{i,j+\frac{1}{2}}-\hat{(v\omega)}^n_{i,j-\frac{1}{2}}],
\end{equation}
involving a numerical flux $\hat{(u\omega)}^n_{i+\frac{1}{2},j}$ and $\hat{(v\omega)}^n_{i,j+\frac{1}{2}}$ as local functions of $u^n_{kl}$,  $v^n_{kl}$ and $\omega^n_{kl}$.  The numerical flux is defined as
\begin{equation}\label{flux}
\begin{aligned}
\hat{(u\omega)}_{i+\frac{1}{2},j} = \frac{1}{2}\left([W_{1y}(\mathbf{u} \circ\boldsymbol{\omega})]_{ij}+[W_{1y}(\mathbf{u} \circ\boldsymbol{\omega})]_{i+1,j}\right),\\
\hat{(v\omega)}_{i,j+\frac{1}{2}} = \frac{1}{2}\left([W_{1x}(\mathbf v\circ\boldsymbol{\omega})]_{ij}+[W_{1x}(\mathbf{v} \circ\boldsymbol{\omega})]_{i,j+1}\right).
\end{aligned}
\end{equation}
Similarly we denote
\begin{equation}\label{uv-flux}
\begin{aligned}
\hat{u}_{i+\frac{1}{2},j} = \frac{1}{2}\left(\left(W_{1y}\mathbf{u}\right)_{ij}+\left(W_{1y}\mathbf{u}\right)_{i+1,j}\right),\\
\hat{v}_{i,j+\frac{1}{2}} = \frac{1}{2}\left(\left(W_{1x}\mathbf{v}\right)_{ij}+\left(W_{1x}\mathbf{v}\right)_{i,j+1}\right).
\end{aligned}
\end{equation}

The limiting is defined in a dimension by dimension manner.
For the flux splitting, it is done as in one-dimension.  
Consider a splitting of $u$ satisfying
 \begin{equation}\label{flux-splitting}
u^{+}\geq 0,  \quad u^{-}\leq 0.
 \end{equation}
 The simplest such  splitting is the Lax-Friedrichs splitting
$$u^{\pm}=\frac12(u\pm\alpha ),\quad \alpha=\max\limits_{(x,y)\in \Omega}|u(x,y)|.$$
 
Then we have 
\begin{equation*}
u=u^{+}+u^{-}, \quad u\omega=u^{+}\omega+u^{-}\omega,
\end{equation*}
and we  write the flux $\hat{(u\omega)}_{i+\frac{1}{2},j} $ and $\hat{u}_{i+\frac{1}{2},j} $ as 
\begin{equation*}
\hat{(u\omega)}_{i+\frac{1}{2},j} =\hat{(u\omega)}^+_{i+\frac{1}{2},j} +\hat{(u\omega)}^-_{i+\frac{1}{2},j}, \quad 
\hat{u}_{i+\frac{1}{2},j}  = \hat{u}^+_{i+\frac{1}{2},j} + \hat{u}^-_{i+\frac{1}{2},j} 
\end{equation*}
where $\hat{(u\omega)}^{\pm}_{i+\frac{1}{2},j}$ and $\hat{u}^{\pm}_{i+\frac{1}{2},j} $ are obtained by adding superscripts $\pm$ to $u_{ij}$ in \eqref{flux} and \eqref{uv-flux} respectively, i.e. 
\begin{align*}
\hat{(u\omega)}^{\pm}_{i+\frac{1}{2},j} = & \frac{1}{2}\left([W_{1y}(\mathbf{u}^{\pm} \circ\boldsymbol{\omega})]_{ij}+[W_{1y}(\mathbf{u}^{\pm} \circ\boldsymbol{\omega})]_{i+1,j}\right), \\
\hat{u}^{\pm}_{i+\frac{1}{2},j} = & \frac{1}{2}\left(\left(W_{1y}\mathbf{u}^{\pm}\right)_{ij}+\left(W_{1y}\mathbf{u}^{\pm}\right)_{i+1,j}\right),
\end{align*}
where $\mathbf{u}^{\pm} = (u^{\pm}_{ij})$. With a dummy index $j$ referring $y$ value, we first take the differences between the high-order numerical flux and the first-order upwind flux
\begin{equation}
d \hat{(u\omega)}_{i+\frac{1}{2}, j}^{+}=\hat{(u\omega)}_{i+\frac{1}{2}, j}^{+}-u_{i+\frac{1}{2}, j}^{+}\bar{\omega}_{i j} , \quad d \hat{(u\omega)}_{i+\frac{1}{2}, j}^{-}=u^{-}_{i+\frac{1}{2}, j}\bar{\omega}_{i+1, j}-\hat{(u\omega)}_{i+\frac{1}{2}, j}^{-}.
\end{equation}

Limit them by 
\begin{equation}\label{tvb-limiter}
\begin{aligned}
d \hat{(u\omega)}_{i+\frac{1}{2}, j}^{+(m)}\, = \,& m\left(d \hat{(u\omega)}_{i+\frac{1}{2}, j}^{+}, \,u_{i+\frac{1}{2}, j}^{+}\Delta_{+}^{x} \bar{\omega}_{i j}, \,u_{i-\frac{1}{2}, j}^{+}\Delta_{+}^{x} \bar{\omega}_{i-1, j}\right), \\
d \hat{(u\omega)}_{i+\frac{1}{2}, j}^{-(m)}\, = \,& m\left(d \hat{(u\omega)}_{i+\frac{1}{2}, j}^{-}, \,u_{i+\frac{1}{2}, j}^{-}\Delta_{+}^{x} \bar{\omega}_{i j}, \,u_{i+\frac{3}{2}, j}^{-}\Delta_{+}^{x} \bar{\omega}_{i+1, j}\right),
\end{aligned}
\end{equation}
where $\Delta_{+}^{x} v_{i j} \equiv v_{i+1, j}-v_{i j}$ is the forward difference operator in the $x$ direction,  and $m$ is the standard $minmod$ function 
\begin{equation}
\label{standard-minmod}
m(a_1,\dots,a_k)=\left\{\begin{array}{ll}
 s \min_{1\leq i\leq k}|a_i|, & \textrm{if $sign(a_1)=\cdots=sign(a_k)=s$,}\\
 0, & \textrm{otherwise.}
\end{array}\right.
\end{equation}

As mentioned in \cite{cockburn1994nonlinearly},  the limiting defined in \eqref{tvb-limiter} maintains the formal accuracy of the compact schemes in smooth regions of the solution with the assumption
\begin{equation}\label{avg-asmp}
\bar{\omega}_{ij} = (W_{1x}W_{1y}\mathbf{\omega})_{ij} = \omega_{ij}+\mathcal O\left(\Delta x^{2}\right) \text{ for } \omega \in C^2.
\end{equation} 
Under the assumption \eqref{avg-asmp}, by simple Taylor expansion, 
\begin{equation}\label{tvb-limiter-te}
\begin{aligned}
d \hat{(u\omega)}_{i+\frac{1}{2}, j}^{\pm} = & \frac{1}{2}u_{i+\frac{1}{2}, j}^{\pm}\omega_{x,ij}\Delta x+\mathcal O\left(\Delta x^{2}\right), \\
\,u_{k+\frac{1}{2}, j}^{\pm}\Delta_{+}^{x} \bar{\omega}_{k j}= & u_{i+\frac{1}{2}, j}^{\pm}\omega_{x,ij}\Delta x+\mathcal O\left(\Delta x^{2}\right),\quad k =i-1,i, i+1.
\end{aligned}
\end{equation}
Hence in smooth regions away from critical points of $\omega$,  for sufficiently small $\Delta x$, the minmod function \eqref{standard-minmod} will pick the first argument, yielding
$$
d \hat{(u\omega)}_{i+\frac{1}{2}, j}^{\pm(m)} = d \hat{(u\omega)}_{i+\frac{1}{2}, j}^{\pm}.
$$
Since the accuracy may degenerate to first-order at critical points,  as a remedy,  the modified $minmod$ function \cite{shu1987tvb,  cockburn1989tvb} is introduced
\begin{equation}
\label{minmod}
\tilde{m}(a_1,\dots,a_k)=\left\{\begin{array}{ll}
 a_1, & \textrm{if $|a_1|\leq P\Delta x^2$,}\\
 m(a_1,\dots,a_k), & \textrm{otherwise,}
\end{array}\right.
\end{equation}
where $P$ is a positive constant independent of $\Delta x$ and $m$ is the standard $minmod$ function \eqref{standard-minmod}. See more detailed discussion in  \cite{cockburn1994nonlinearly}.

Then we obtain the limited numerical fluxes as
\begin{equation}\label{modified-flux}
\hat{(u\omega)}_{i+\frac{1}{2}, j}^{+(m)} = u_{i+\frac{1}{2}, j}^{+}\bar{\omega}_{i j} + d \hat{(u\omega)}_{i+\frac{1}{2}, j}^{+(m)}, \quad 
\hat{(u\omega)}_{i+\frac{1}{2}, j}^{-(m)} = u^{-}_{i+\frac{1}{2}, j}\bar{\omega}_{i+1, j} - d \hat{(u\omega)}_{i+\frac{1}{2}, j}^{-(m)}.
\end{equation}
and
\begin{equation}\label{modified-flux-1}
\hat{(u\omega)}_{i+\frac{1}{2}, j}^{(m)} = \hat{(u\omega)}_{i+\frac{1}{2}, j}^{+(m)} + \hat{(u\omega)}_{i+\frac{1}{2}, j}^{-(m)} 
\end{equation}
The flux in the y-direction can be defined analogously.

The following result was proven in \cite{cockburn1994nonlinearly}:
\begin{lemma}
  For any $n$ and $\Delta t$ such that $0\leq n\Delta t\leq T$, scheme \eqref{incomp-conservationform} with flux \eqref{modified-flux-1} satisfies a maximum principle in the means:
$$
\max _{i, j}\left|\bar{\omega}_{i j}^{n+1}\right| \leq \max _{i, j}\left|\bar{\omega}_{i j}^{n}\right|
$$
under the CFL condition
$$
\left[\max \left(u^{+}\right)+\max \left(-u^{-}\right)\right] \frac{\Delta t}{\Delta x}+\left[\max \left(v^{+}\right)+\max \left(-v^{-}\right)\right] \frac{\Delta t}{\Delta y} \leq \frac{1}{2}
$$
where the maximum is taken in $\min_{i, j} u_{i j}^{n} \leq u \leq \max_{i, j} u_{i j}^{n}$, $\min_{i, j} v_{i j}^{n} \leq v \leq \max_{i, j} v_{i j}^{n}$.
 
\end{lemma}

\subsection{The bound-preserving property of the nonlinear scheme with modified flux}

 The compact finite difference scheme with the TVB limiter in the last section is 
 \begin{equation}\label{nonlinear-scheme}
  \bar{\omega}_{ij}^{n+1}=\bar{\omega}_{ij}^{n}-\lambda_1 \left( \hat{(u\omega)}_{i+\frac{1}{2}, j}^{(m)} -\hat{(u\omega)}_{i-\frac{1}{2}, j}^{(m)} \right) - \lambda_2 \left( \hat{(v\omega)}_{i, j+\frac{1}{2}}^{(m)} -\hat{(v\omega)}_{i, j-\frac{1}{2}}^{(m)} \right),
 \end{equation}
where the numerical flux $\hat{(u\omega)}_{i+\frac{1}{2}, j}^{(m)}$, $\hat{(u\omega)}_{i, j+\frac{1}{2}}^{(m)}$ is the modified flux approximating \eqref{flux}.

\begin{thm}\label{tvb-avg-bp}
 If $\omega^n_{ij}\in[m, M]$,  under the CFL condition
 \begin{equation}\label{cfl-cond}
\lambda_1 \max _{i,j}\left|u^{(\pm)}_{ij}\right| \leq \frac{1}{24}, \quad \lambda_2 \max _{i,j}\left|v^{(\pm)}_{ij}\right| \leq \frac{1}{24},
\end{equation}
 the nonlinear scheme \eqref{nonlinear-scheme} satisfies
 \[  \bar{\omega}_{ij}^{n+1} \in \left[m, M\right].\]
 \end{thm}

\begin{proof}
We have
\begin{equation} \label{4terms}
\begin{aligned}
  & \bar{\omega}_{ij}^{n+1}   =  \bar{\omega}_{ij}^{n}-\lambda_1 \left( \hat{(u\omega)}_{i+\frac{1}{2}, j}^{(m)} -\hat{(u\omega)}_{i-\frac{1}{2}, j}^{(m)} \right) - \lambda_2 \left( \hat{(v\omega)}_{i, j+\frac{1}{2}}^{(m)} -\hat{(v\omega)}_{i, j-\frac{1}{2}}^{(m)} \right) \\
  = & \frac18\left((\bar{\omega}^n_{ij}-8\lambda_1\hat{(u\omega)}_{i+\frac{1}{2}, j}^{+(m)})+(\bar{\omega}^n_{ij}-8\lambda_1\hat{(u\omega)}_{i+\frac{1}{2}, j}^{-(m)})  +
(\bar{\omega}^n_{ij}+8\lambda_1\hat{(u\omega)}_{i-\frac{1}{2}, j}^{+(m)}) + 
(\bar{\omega}^n_{ij}+8\lambda_1\hat{(u\omega)}_{i-\frac{1}{2}, j}^{-(m)})\right.\\
+ &\left.(\bar{\omega}^n_{ij}-8\lambda_2\hat{(v\omega)}_{i, j+\frac{1}{2}}^{+(m)})+
(\bar{\omega}^n_{ij}-8\lambda_2\hat{(v\omega)}_{i, j+\frac{1}{2}}^{-(m)})  +
(\bar{\omega}^n_{ij}+8\lambda_2\hat{(v\omega)}_{i, j-\frac{1}{2}}^{+(m)}) + 
(\bar{\omega}^n_{ij}+8\lambda_2\hat{(v\omega)}_{i, j-\frac{1}{2}}^{-(m)})\right).
\end{aligned}
\end{equation} 

Under the CFL condition \eqref{cfl-cond}, we will prove that the eight terms satisfy the following bounds
\begin{equation}\label{bp-result}
  \begin{aligned}
  & \bar{\omega}^n_{ij}-8\lambda_1\hat{(u\omega)}_{i+\frac{1}{2}, j}^{+(m)}\in\left[m-8\lambda_1\hat{u}_{i+\frac{1}{2}, j}^{+}m,  M-8\lambda_1\hat{u}_{i+\frac{1}{2}, j}^{+}M\right],\\
  &\bar{\omega}^n_{ij}-8\lambda_1\hat{(u\omega)}_{i+\frac{1}{2}, j}^{-(m)}\in\left[m-8\lambda_1\hat{u}_{i+\frac{1}{2}, j}^{-}m,  M-8\lambda_1\hat{u}_{i+\frac{1}{2}, j}^{-}M\right],\\
  &\bar{\omega}^n_{ij}+8\lambda_1\hat{(u\omega)}_{i-\frac{1}{2}, j}^{+(m)}\in\left[m+8\lambda_1\hat{u}_{i-\frac{1}{2}, j}^{+}m,  M+8\lambda_1\hat{u}_{i-\frac{1}{2}, j}^{+}M\right],\\
    & \bar{\omega}^n_{ij}+8\lambda_1\hat{(u\omega)}_{i-\frac{1}{2}, j}^{-(m)}\in \left[m+8\lambda_1\hat{u}_{i-\frac{1}{2}, j}^{-}m,  M+8\lambda_1\hat{u}_{i-\frac{1}{2}, j}^{-}M\right],\\
    & \bar{\omega}^n_{ij}-8\lambda_2\hat{(v\omega)}_{i, j+\frac{1}{2}}^{+(m)}\in\left[m-8\lambda_2\hat{v}_{i, j+\frac{1}{2}}^{+}m,  M-8\lambda_2\hat{v}_{i, j+\frac{1}{2}}^{+}M\right],\\
    & \bar{\omega}^n_{ij}-8\lambda_2\hat{(v\omega)}_{i, j+\frac{1}{2}}^{-(m)}\in\left[m-8\lambda_2\hat{v}_{i, j+\frac{1}{2}}^{-}m,  M-8\lambda_2\hat{v}_{i, j+\frac{1}{2}}^{-}M\right],\\
    & \bar{\omega}^n_{ij}+8\lambda_2\hat{(v\omega)}_{i, j-\frac{1}{2}}^{+(m)}\in\left[m+8\lambda_2\hat{v}_{i, j-\frac{1}{2}}^{+}m,  M+8\lambda_2\hat{v}_{i, j-\frac{1}{2}}^{+}M\right],\\
    & \bar{\omega}^n_{ij}+8\lambda_2\hat{(v\omega)}_{i, j-\frac{1}{2}}^{-(m)}\in\left[m+8\lambda_2\hat{v}_{i, j-\frac{1}{2}}^{-}m,  M+8\lambda_2\hat{v}_{i, j-\frac{1}{2}}^{-}M\right].
   \end{aligned}
   \end{equation}
   
For \eqref{bp-result},  by taking the sum of the lower bounds and upper bounds and multiplying them by $\frac{1}{8}$, we obtain
\begin{equation}
\bar{\omega}^{n+1}_{ij} \in \left[ m - mO_{ij},  M - MO_{ij} \right],
\end{equation}
  with 
  \begin{equation}
  \begin{aligned}
  O_{ij} = &  \lambda_1(\hat{u}_{i+\frac{1}{2}, j} - \hat{u}_{i-\frac{1}{2}, j}) - \lambda_2(\hat{u}_{i, j+\frac{1}{2}} - \hat{u}_{i, j-\frac{1}{2}}) \\
  = &  \frac{\lambda_1}{2}\left((W_{1y}\mathbf{u})_{i+1,j} - (W_{1y}\mathbf{u})_{i-1,j})\right) +  \frac{\lambda_2}{2}\left((W_{1y}\mathbf{v})_{i,j+1} - (W_{1y}\mathbf{v})_{i,j-1})\right) \\
  = & \frac{\Delta t}{2}(D_{x}W_{1y}\mathbf{u} + D_{y}W_{1x}\mathbf{v}) =  0.
  \end{aligned}
  \end{equation}
  Therefore, we conclude $\bar{\omega}^{n+1}_{ij} \in \left[m, M \right]$.

  We only discuss the first two term in \eqref{bp-result} since the proof for the rest is similar. By the definition of the modified $minmod$ function \eqref{minmod} and \eqref{modified-flux},  we have
\begin{equation}
\label{flux-range}
\begin{aligned}
\hat{(u\omega)}_{i+\frac{1}{2}, j}^{+(m)} \in & \left[\min\{\hat{(u\omega)}_{i+\frac{1}{2}, j}^{+}, u_{i+\frac{1}{2}, j}^{+}\bar{\omega}_{i j}\} , \max\{\hat{(u\omega)}_{i+\frac{1}{2}, j}^{+}, u_{i+\frac{1}{2}, j}^{+}\bar{\omega}_{i j}\}\right], \\
\hat{(u\omega)}_{i+\frac{1}{2}, j}^{-(m)} \in & \left[\min\{\hat{(u\omega)}_{i+\frac{1}{2}, j}^{-}, u_{i+\frac{1}{2}, j}^{-}\bar{\omega}_{i+1, j}\} , \max\{\hat{(u\omega)}_{i+\frac{1}{2}, j}^{-}, u_{i+\frac{1}{2}, j}^{-}\bar{\omega}_{i+1,j}\}\right].
\end{aligned}
\end{equation}

We notice that under CFL condition \eqref{cfl-cond}, 
\begin{equation}
\bar{\omega}^n_{ij}-8\lambda_1\hat{(u\omega)}_{i+\frac{1}{2}, j}^{+}, \quad \bar{\omega}^n_{ij}-8\lambda_1u_{i+\frac{1}{2}, j}^{+}\bar{\omega}_{i j}^n,\quad
\bar{\omega}^n_{ij}-8\lambda_1\hat{(u\omega)}_{i+\frac{1}{2}, j}^{-}
\end{equation}
are all monotonically increasing functions with respect to variables $\omega^n_{kj}$, $k = i-1,i,i+1$.  Due to the flux splitting \eqref{flux-splitting},
\begin{equation}
 \bar{\omega}^n_{ij}-8\lambda_1u_{i+\frac{1}{2}, j}^{-}\bar{\omega}_{i+1, j}^n
\end{equation}
is also a monotonically increasing function with respect to variables $\omega^n_{kj}$, $k = i-1,i,i+1, i+2$. 
Therefore, with the assumption $\omega^n_{ij} \in [m, M]$, we obtain
\begin{equation}\label{lubound-range}
\begin{aligned}
\bar{\omega}^n_{ij}-8\lambda_1\hat{(u\omega)}_{i+\frac{1}{2}, j}^{+}, \quad \bar{\omega}^n_{ij}-8\lambda_1u_{i+\frac{1}{2}, j}^{+}\bar{\omega}_{i j}^n \in \left[m-8\lambda_1\hat{u}_{i+\frac{1}{2}, j}^{+}m,  M-8\lambda_1\hat{u}_{i+\frac{1}{2}, j}^{+}M\right],\\
\bar{\omega}^n_{ij}-8\lambda_1\hat{(u\omega)}_{i+\frac{1}{2}, j}^{-}, \quad \bar{\omega}^n_{ij}-8\lambda_1u_{i+\frac{1}{2}, j}^{-}\bar{\omega}_{i+1, j}^n \in \left[m-8\lambda_1\hat{u}_{i+\frac{1}{2}, j}^{-}m,  M-8\lambda_1\hat{u}_{i+\frac{1}{2}, j}^{-}M\right],
\end{aligned}
\end{equation}
with \eqref{flux-range} ,  which implies the first two terms of \eqref{bp-result}.
\end{proof}
\begin{rmk}\label{rmk-tvb}
We remark here the above proof is independent of the second and third arguments of the $minmod$ function \eqref{minmod}. Therefore,  the proof hold for other limiters with different second and third arguments in the same minmod function \eqref{minmod}.
\end{rmk}
\begin{rmk}
The TVB limiter in this paper is designed to modify
 the convection flux only thus it also applies to the  Navier-Stokes equation. Moreover, under suitable CFL condition, the full scheme with TVB limiter can still preserve $\tilde{\bar{\omega}}_{ij}^{n+1} \in[m,M]$ with $\omega_{ij}^{n} \in[m,M]$.
\end{rmk}

\subsection{An alternative TVB limiter}\label{sec-tvb-limiter-def-2}
Another TVB limiter can be defined by replacing \eqref{tvb-limiter} with
\begin{equation}\label{tvb-limiter-1}
\begin{aligned}
d \hat{(u\omega)}_{i+\frac{1}{2}, j}^{+(m)}\, = \,& m\left(d \hat{(u\omega)}_{i+\frac{1}{2}, j}^{+}, \, \Delta_x^+(u_{i+\frac{1}{2}, j}^{+}\bar{\omega}_{i j}), \, \Delta_x^+(u_{i-\frac{1}{2}, j}^{+}\bar{\omega}_{i-1, j})\right), \\
d \hat{(u\omega)}_{i+\frac{1}{2}, j}^{-(m)}\, = \,& m\left(d \hat{(u\omega)}_{i+\frac{1}{2}, j}^{-}, \,\Delta_x^+(u^{-}_{i-\frac{1}{2}, j}\bar{\omega}_{i, j}),\,\Delta_x^+(u^{-}_{i+\frac{1}{2}, j}\bar{\omega}_{i+1, j})\right).
\end{aligned}
\end{equation}
All the other procedures in the limiter are exactly the same as in Section \ref{sec-tvb-limiter-def}. 
The limiter does not affect the bound-preserving property due to the arguments in Remark \ref{rmk-tvb}.

\section{Numerical Tests}
\label{sec-numeric-test}
 In this subsection, we test the fourth-order compact finite difference scheme with both the bound-preserving and the TVB limiter for the two-dimensional incompressible flow.

In the numerical test, we refer to the bound-preserving limiter as BP,  the TVB limiter in Section \ref{sec-tvb-limiter-def} as TVB1, and the TVB limiter in section \ref{sec-tvb-limiter-def-2} as TVB2. The parameter in the minmod function used in TVB limiters is denoted as P. 
In all the following numerical tests, we use SSPRK(3,3) as mentioned in section \ref{sec-time-dis}.
\subsection{Accuracy Test}
For the Euler Equation \eqref{incompressibleEuler} with periodic boundary condition and initial data $\omega(x, y, 0)=-2 \sin (2x) \sin (y)$ on the domain $[0,2 \pi] \times[0,2 \pi]$,  the exact solution is $\omega(x, y, t)=-2 \sin (2x) \sin (y)$.  
We test the accuracy of the proposed scheme on this solution. 
The errors for $P=300$ are given in Table \ref{acc-test-euler-bp-tvb}, and 
we observe the designed order of accuracy for this special steady state solution. 
\begin{table}[htbp]
\centering
\caption{Incompressible Euler equations.  Fourth-order compact FD for vorticity,  $t = 0.5$. With BP and TVB1 limiters, P =300.}
\label{acc-test-euler-bp-tvb}
  \begin{tabular}{|c|cc|cc|}
\hline  $N\times N$ &  $L^2$ error  &  order & $L^\infty$ error & order  \\
\hline
  $32\times 32$  & 3.16E-3 &  - & 1.00E-3 & -\\ 
 \hline
  $64\times 64$  & 1.86E-4 & 4.09 & 5.90E-5 & 4.09\\
\hline
  $128\times 128$  & 1.14E-5 & 4.02 & 3.63E-6 & 4.02\\
\hline
  $256\times 256$ & 7.13E-7 & 4.01 & 2.67E-7 & 4.00\\
\hline
\end{tabular}
\end{table}




\subsection{Double Shear Layer Problem}
We test the scheme for the double shear layer problem on the domain $[0, 2\pi]\times[0, 2\pi]$ with a periodic boundary condition.
The initial condition is
$$\omega(x,y,0)=\left\{\begin{array}{ll}
 \delta cos(x)-\frac{1}{\rho}sech^2((y-\frac{\pi}{2})/\rho), \,y\leq \pi\\
 \delta cos(x)+\frac{1}{\rho}sech^2((\frac{3\pi}{2}-y)/\rho), y>\pi
\end{array}\right.$$
with $\delta = 0.05$ and $\rho = \pi/15$. The vorticity $\omega$ at time $T = 6$ and $T= 8$ are shown in Figure \ref{doubleshearlayer1}, Figure \ref{doubleshearlayer2} and
Figure \ref{doubleshearlayer3}. 
With both the bound-preserving limiter and TVB limiter,
the numerical solutions are ensured to be in the range $[-\delta-\frac{1}{\rho},\delta+\frac{1}{\rho}]$.
The TVB limiter can also reduce oscillations. 


\begin{figure}[ht]
\begin{subfigure}{.49\textwidth}
\includegraphics[scale=0.42]{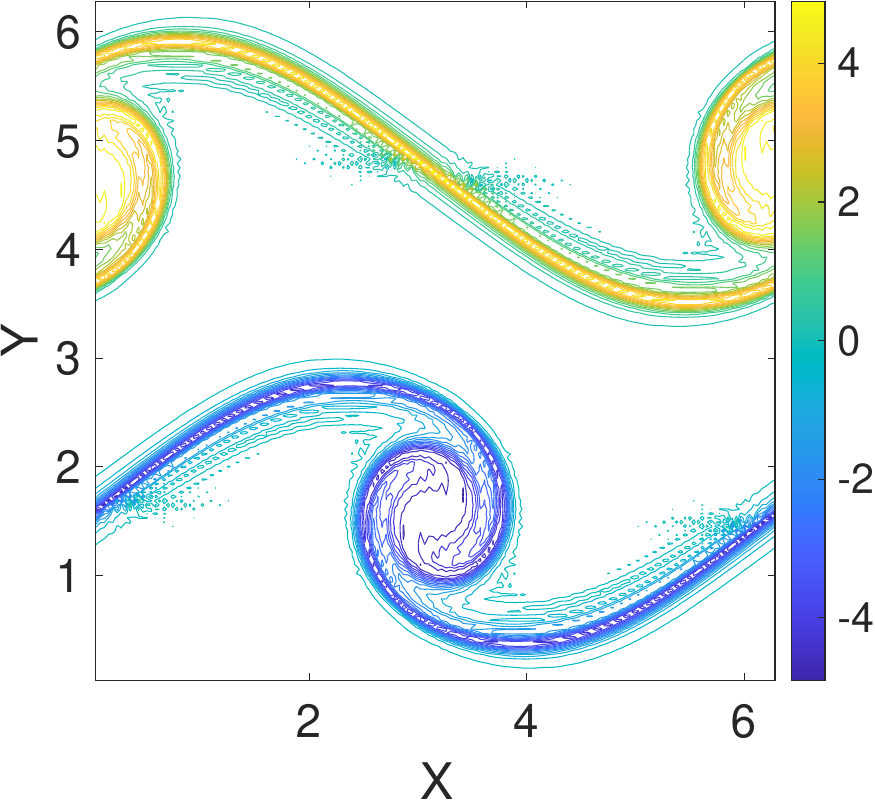}
\caption{$T=6$, with no limiter.}
\end{subfigure}
\begin{subfigure}{.49\textwidth}
\includegraphics[scale=0.42]{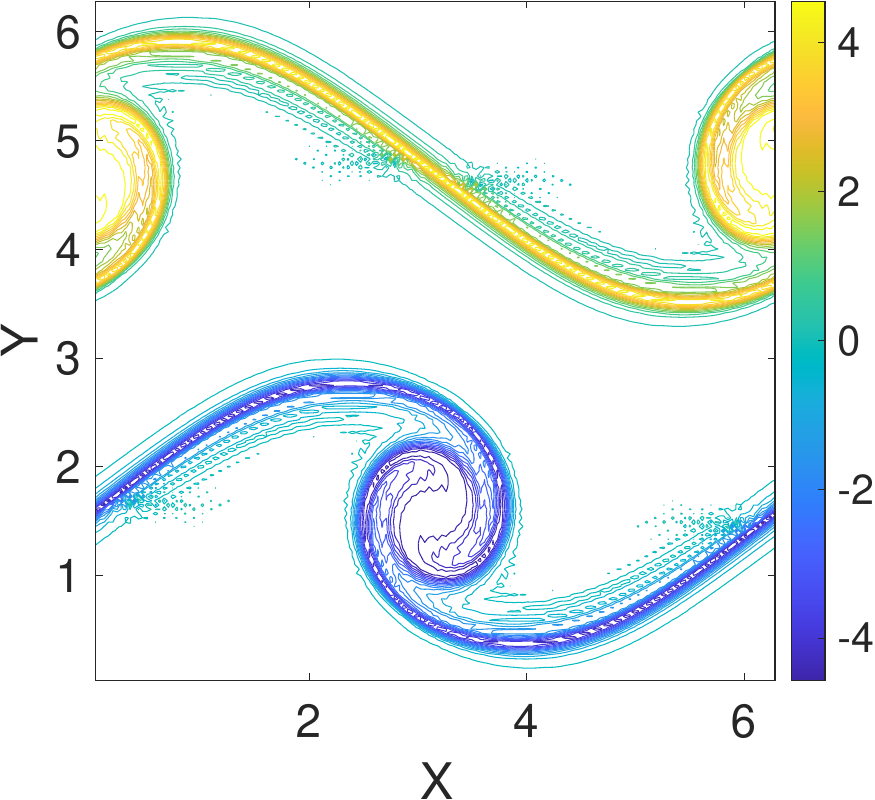}
\caption{$T=6$, with only BP.}
\end{subfigure}

\begin{subfigure}{.49\textwidth}
\includegraphics[scale=0.42]{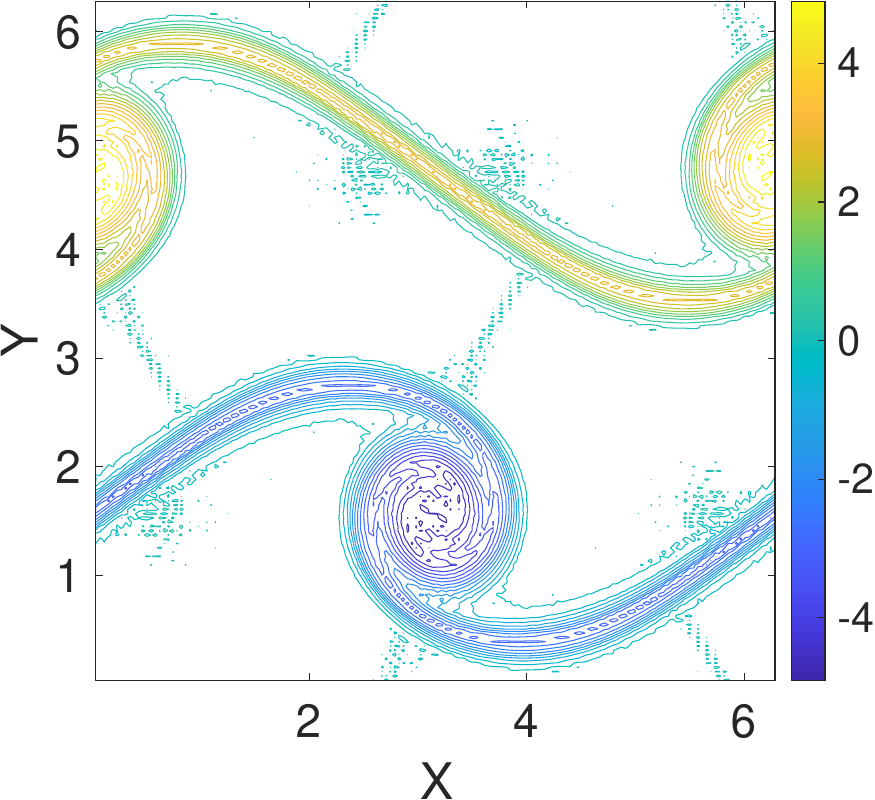}
\caption{$T=6$, with TVB1, P=100.}
\end{subfigure}
\begin{subfigure}{.49\textwidth}
\includegraphics[scale=0.42]{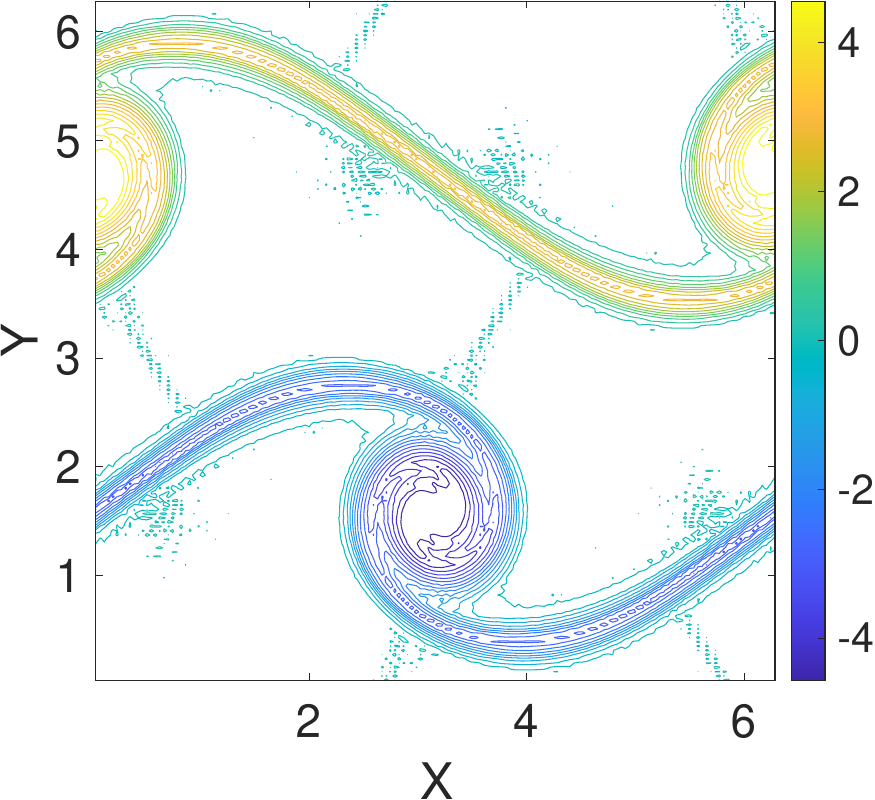}
\caption{$T=6$, with BP and TVB1, P=100.}
\end{subfigure}

\begin{subfigure}{.49\textwidth}
\includegraphics[scale=0.42]{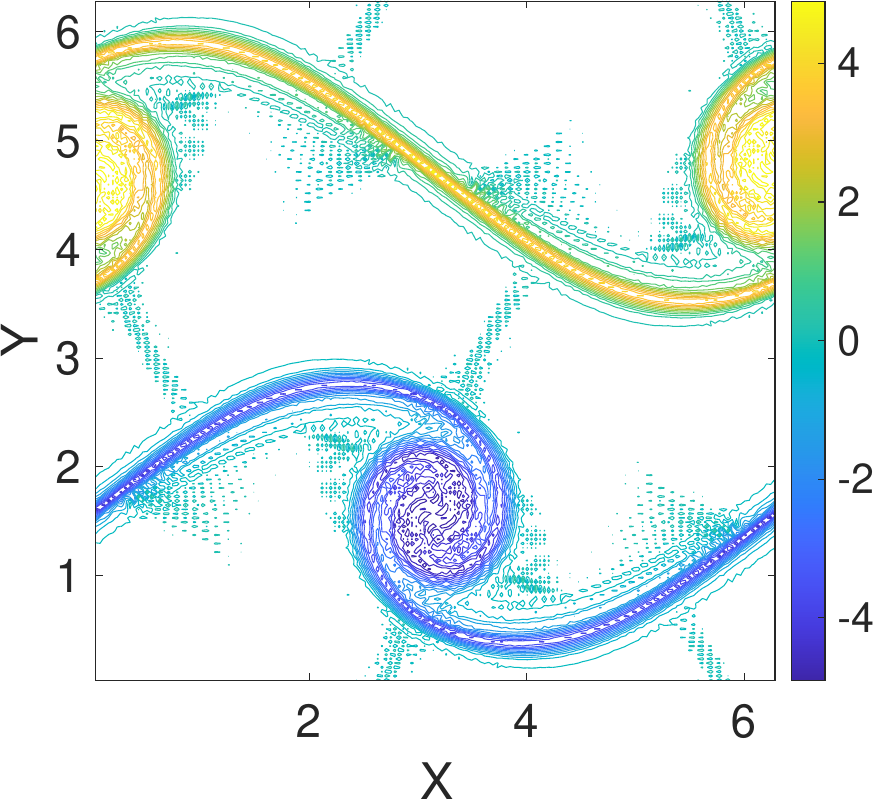}
\caption{$T=6$, with TVB1, P=300.}
\end{subfigure}
\begin{subfigure}{.49\textwidth}
\includegraphics[scale=0.42]{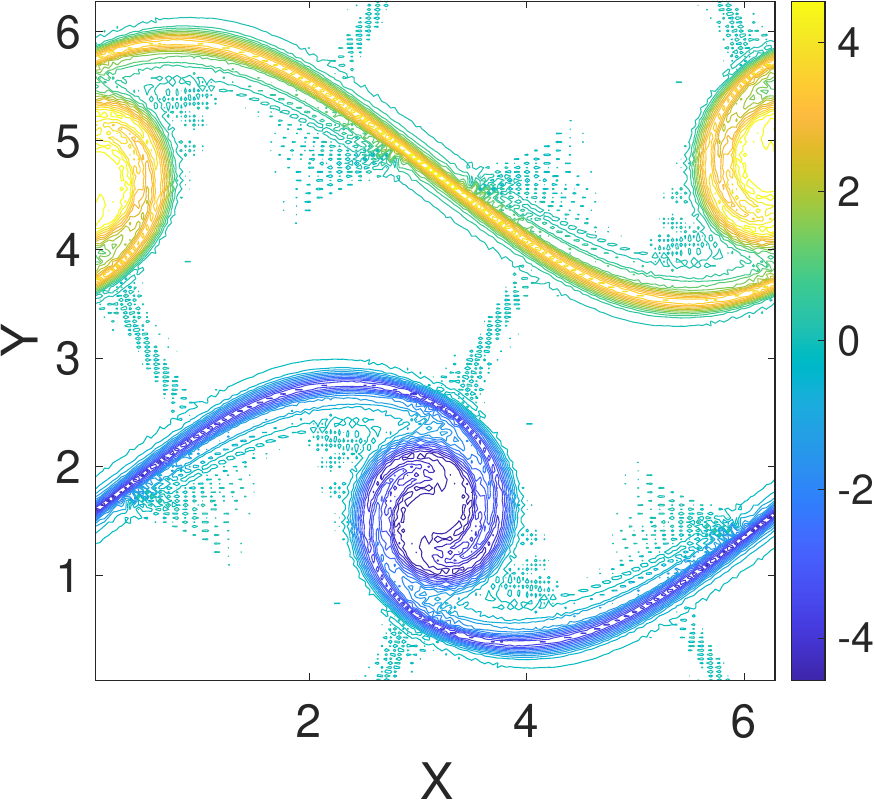}
\caption{$T=6$, with BP and TVB1, P=300.}
\end{subfigure}
\caption{Double shear layer problem. Fourth-order compact finite difference with SSP Runge–Kutta method on a $160\times 160$ mesh 
solving the incompressible Euler equation \eqref{incompressibleEuler}  at $T=6$. The time step is $\Delta t=\frac{1}{24\max_x |\mathbf u_0|}\Delta x$.}
\label{doubleshearlayer1}
\end{figure}

\begin{figure}[ht]
\begin{subfigure}{.49\textwidth}
\includegraphics[scale=0.42]{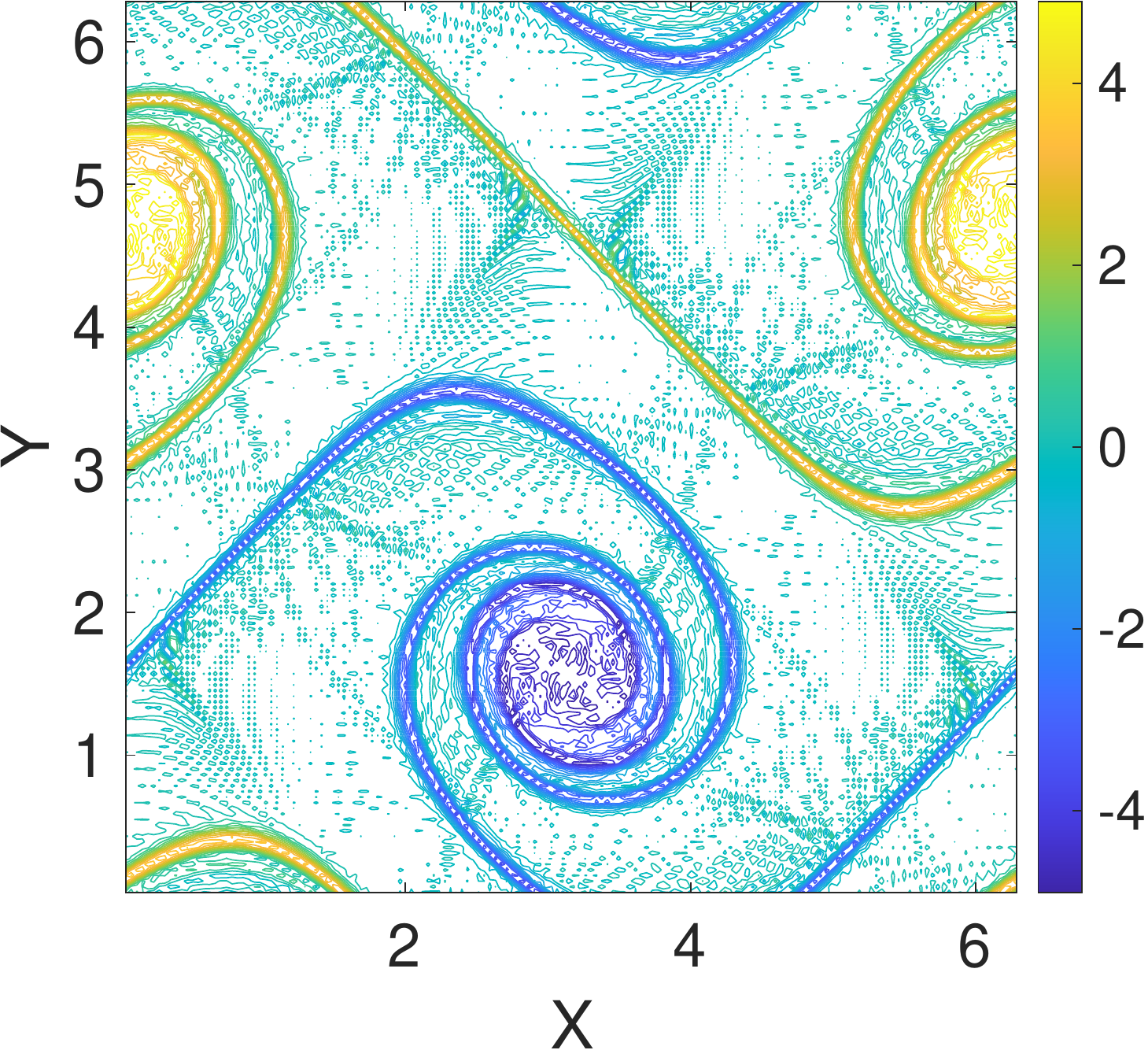}
\caption{$T=8$, with no limiter.}
\end{subfigure}
\begin{subfigure}{.49\textwidth}
\includegraphics[scale=0.42]{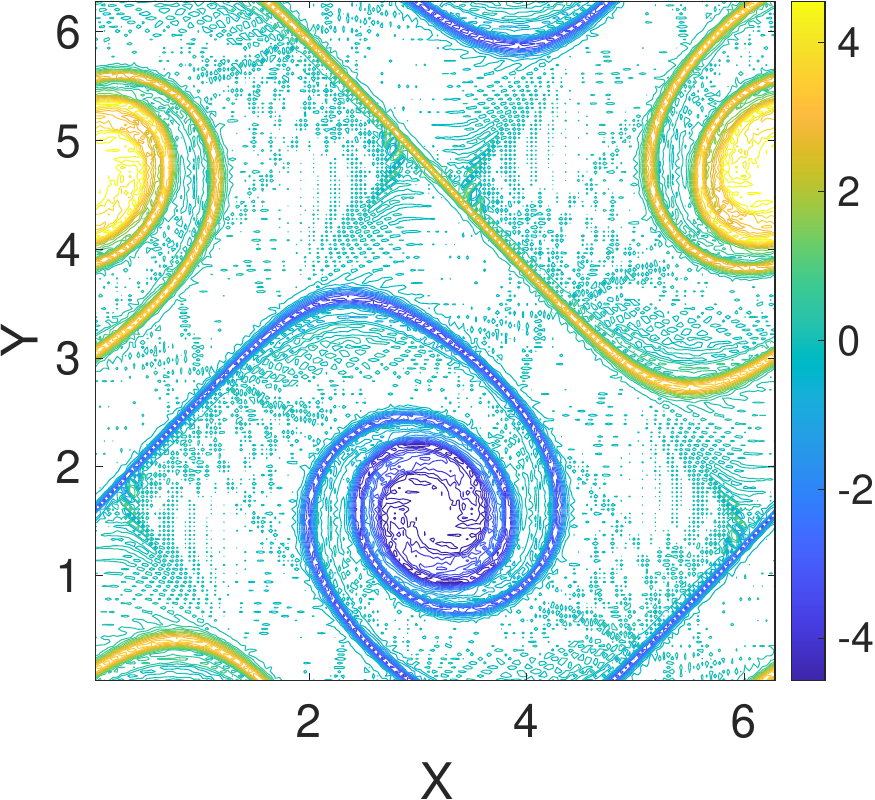}
\caption{$T=8$, with BP.}
\end{subfigure}

\begin{subfigure}{.49\textwidth}
\includegraphics[scale=0.42]{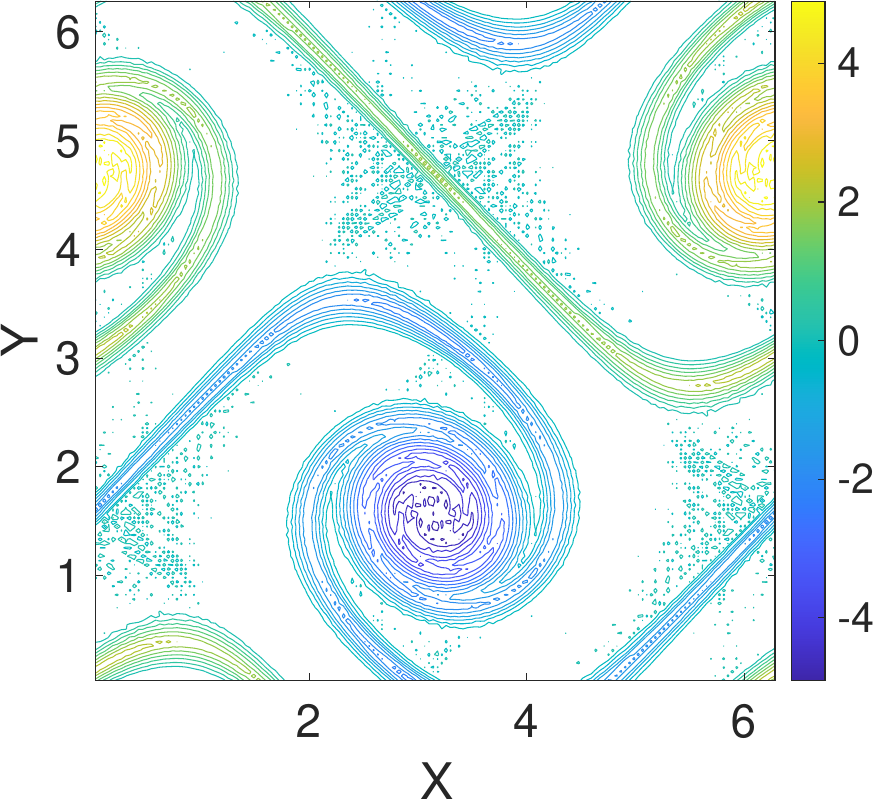}
\caption{$T=8$, with TVB1, P=100.}
\end{subfigure}
\begin{subfigure}{.49\textwidth}
\includegraphics[scale=0.42]{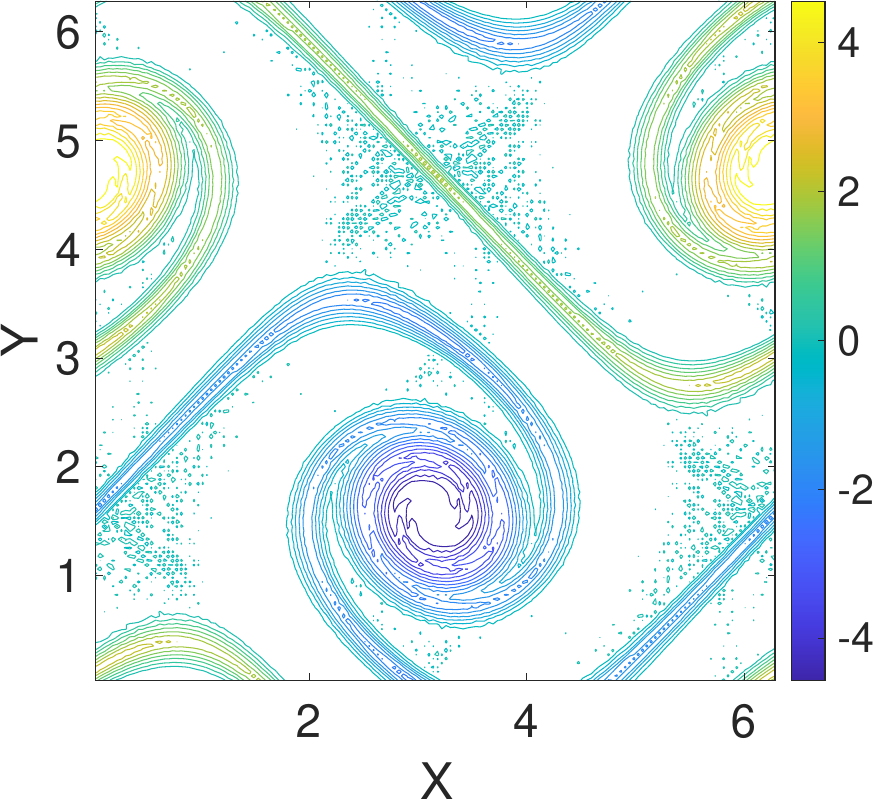}
\caption{$T=8$, with BP and TVB1, P=100.}
\end{subfigure}

\begin{subfigure}{.49\textwidth}
\includegraphics[scale=0.42]{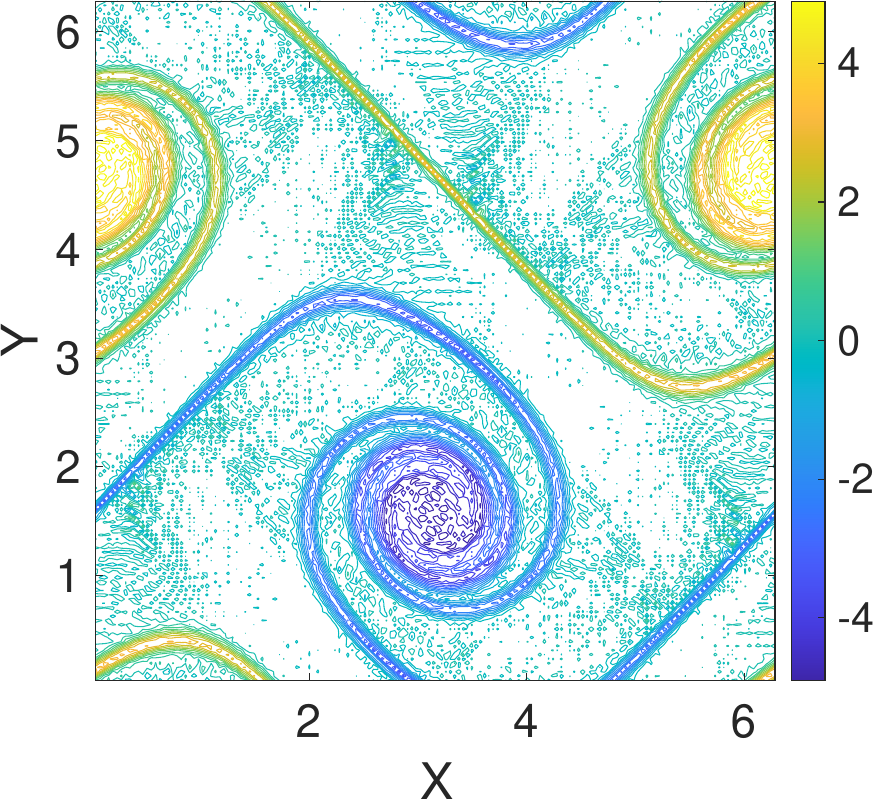}
\caption{$T=8$, with TVB1, P=300.}
\end{subfigure}
\begin{subfigure}{.49\textwidth}
\includegraphics[scale=0.42]{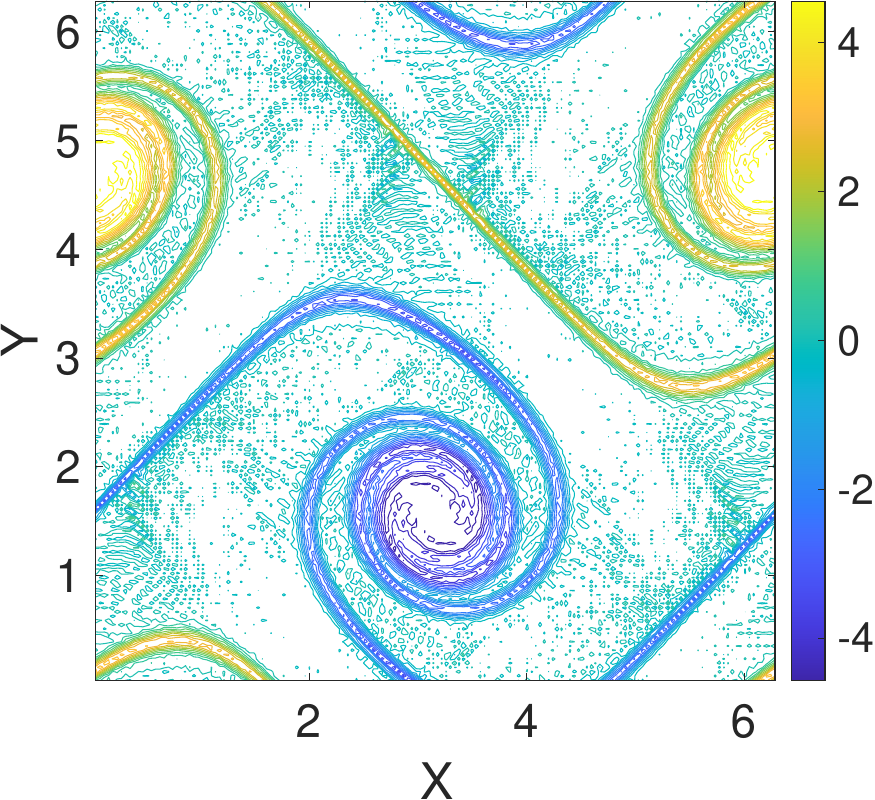}
\caption{$T=8$, with BP and TVB1, P=300.}
\end{subfigure}
\caption{Double shear layer problem. Fourth-order compact finite difference with SSP Runge–Kutta method on a $160\times 160$ mesh 
solving the incompressible Euler equation \eqref{incompressibleEuler}  at $T=8$. The time step is $\Delta t=\frac{1}{24\max_x |\mathbf u_0|}\Delta x$. }
\label{doubleshearlayer2}
\end{figure}

\begin{figure}[ht]
\begin{subfigure}{.49\textwidth}
\includegraphics[scale=0.42]{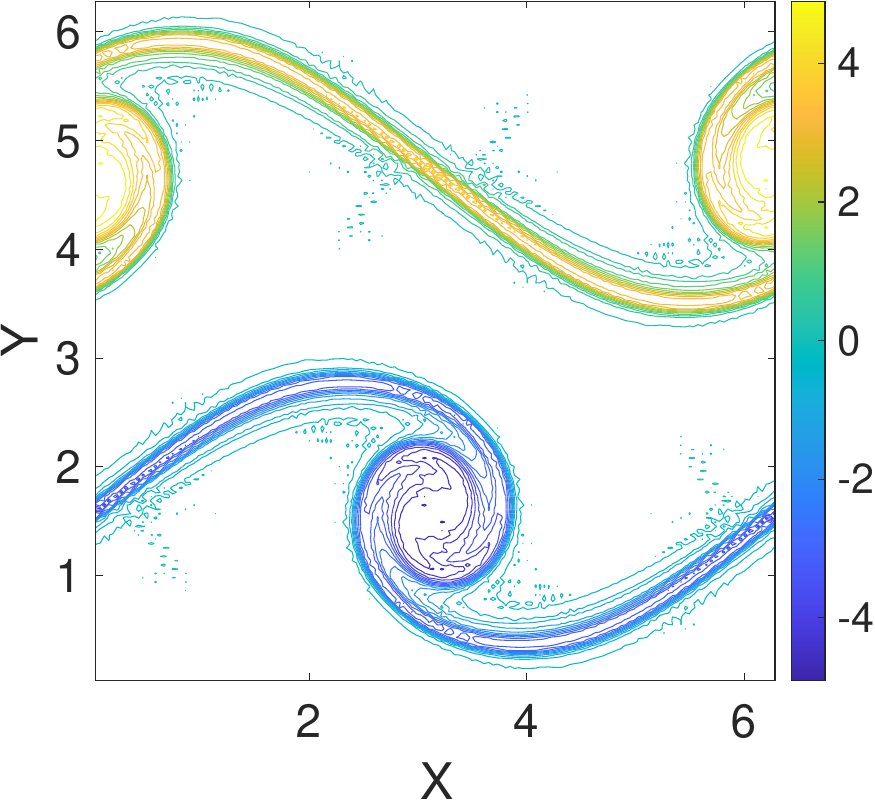}
\caption{$T=6$, without TVB2.}
\end{subfigure}
\begin{subfigure}{.49\textwidth}
\includegraphics[scale=0.42]{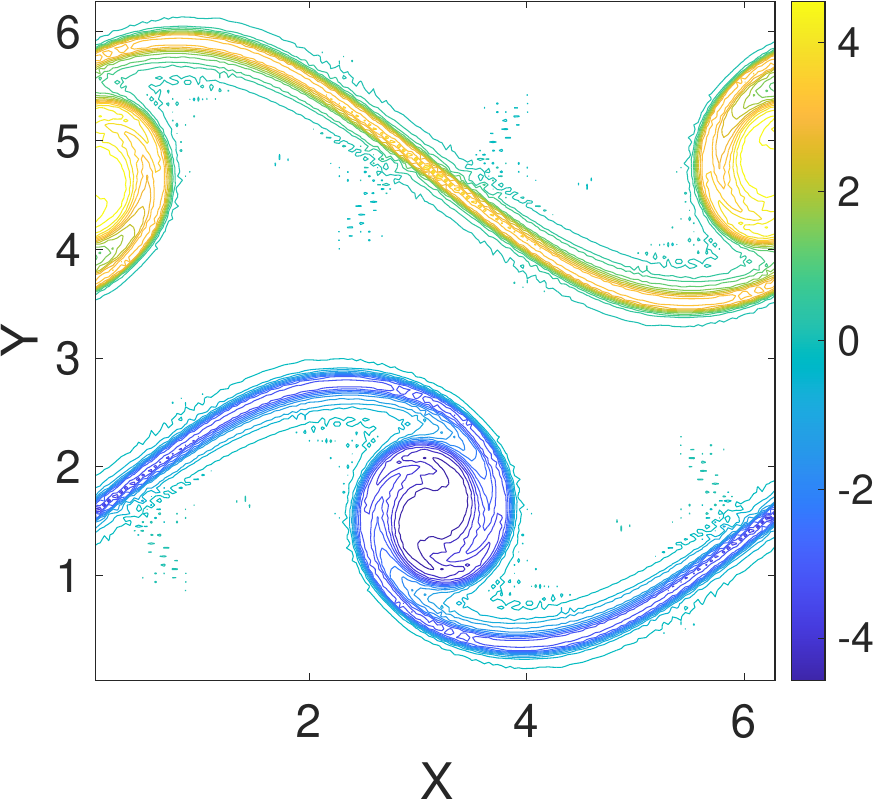}
\caption{$T=6$, with BP and TVB2.}
\end{subfigure}

\begin{subfigure}{.49\textwidth}
\includegraphics[scale=0.42]{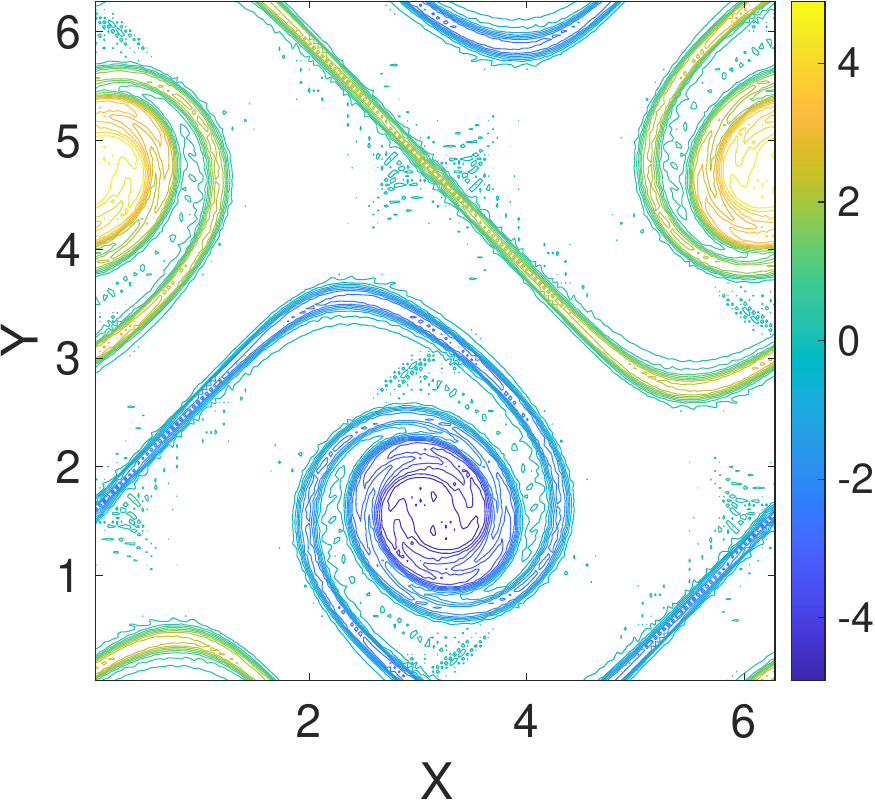}
\caption{$T=8$, with TVB2, P=100.}
\end{subfigure}
\begin{subfigure}{.49\textwidth}
\includegraphics[scale=0.42]{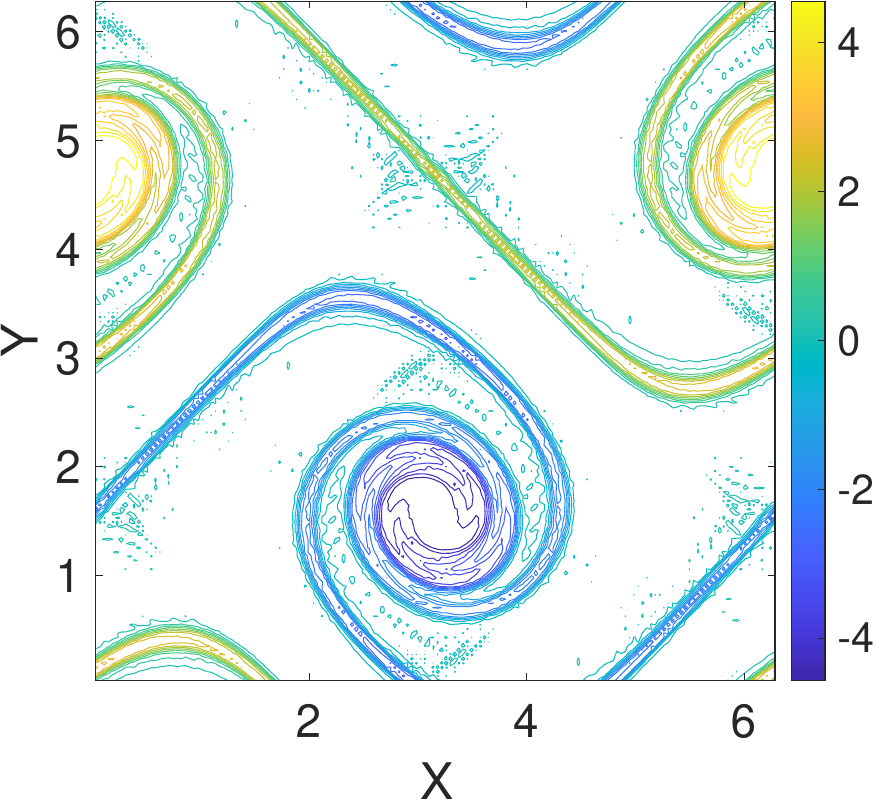}
\caption{$T=8$, with BP and TVB2, P=100.}
\end{subfigure}

\begin{subfigure}{.49\textwidth}
\includegraphics[scale=0.42]{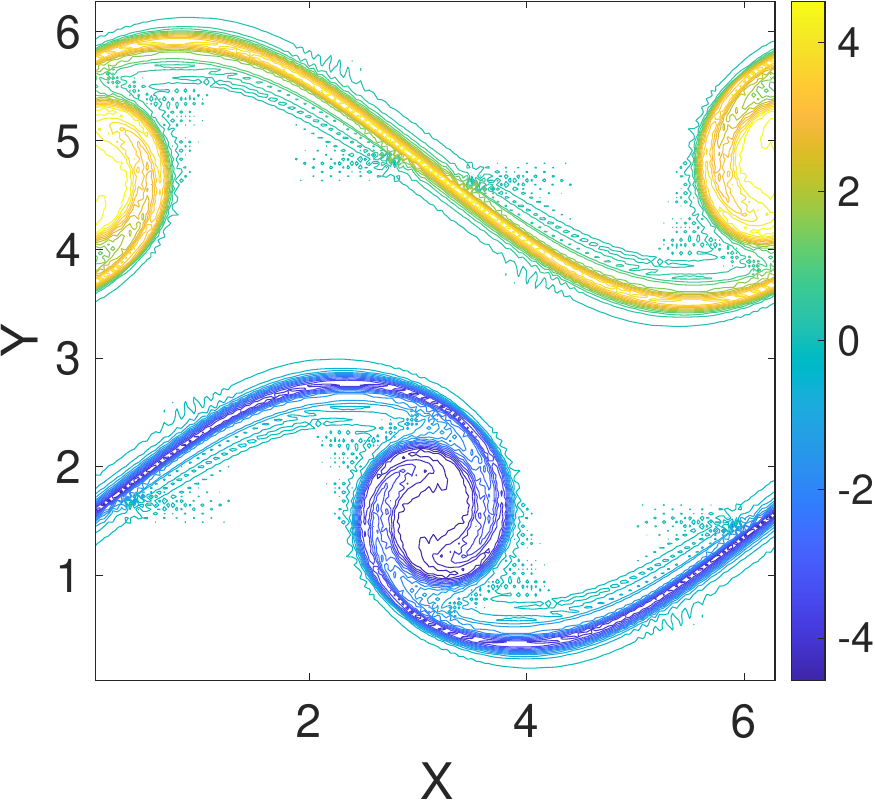}
\caption{$T=6$, with BP and TVB2, P=300.}
\end{subfigure}
\begin{subfigure}{.49\textwidth}
\includegraphics[scale=0.42]{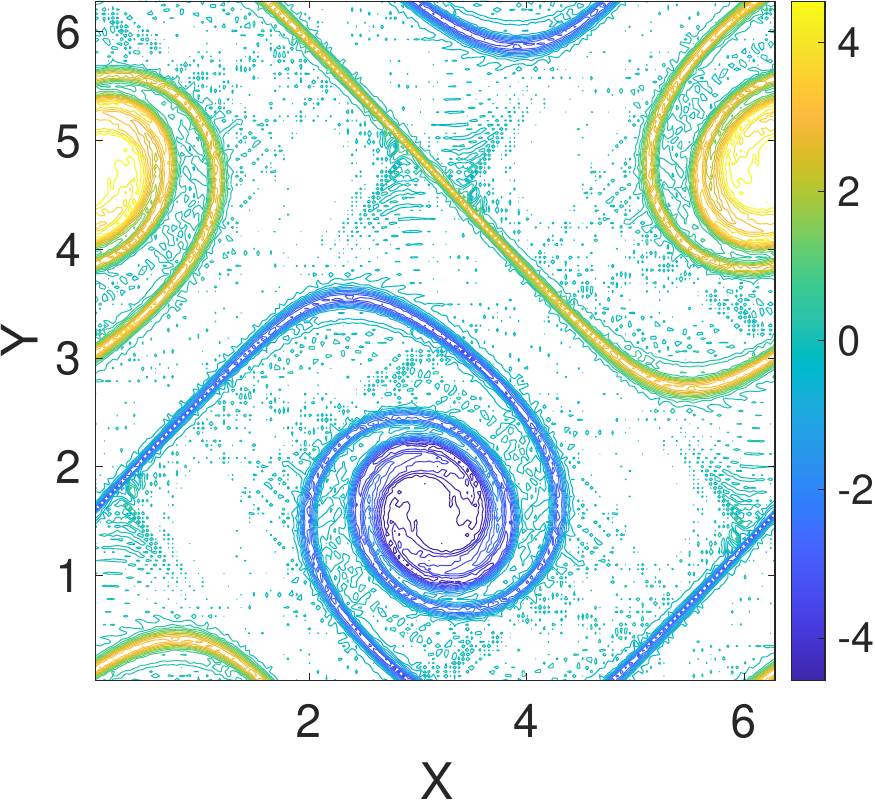}
\caption{$T=8$, with BP and TVB2, P=300.}
\end{subfigure}
\caption{Double shear layer problem. Fourth-order compact finite difference with SSP Runge–Kutta method on a $160\times 160$ mesh 
solving the incompressible Euler equation \eqref{incompressibleEuler}  at $T=6$ and $T=8$. The time step is $\Delta t=\frac{1}{24\max_x |\mathbf u_0|}\Delta x$. }
\label{doubleshearlayer3}
\end{figure}
 
\subsection{Vortex Patch Problem.}
 We test the limiters for the vortex patch problem in the domain $[0, 2\pi]\times[0, 2\pi]$ with a periodic boundary condition.
 The initial condition is
 $$\omega(x,y,0)=\left\{\begin{array}{lll}
-1, &(x,y)\in[\frac{\pi}{2},\frac{3\pi}{2}]\times[\frac{\pi}{4},\frac{3\pi}{4}];\\
1, &(x,y)\in[\frac{\pi}{2},\frac{3\pi}{2}]\times[\frac{5\pi}{4},\frac{7\pi}{4}];\\
0, &\mbox{otherwise.}\end{array}\right.$$
Numerical solutions for incompressible Euler are shown in Figure \ref{incompressibleEulerVP1},  Figure \ref{incompressibleEulerVP2}, Figure \ref{incompressibleEulerVP3} and Figure \ref{incompressibleEulerVP4}.
We can observe that the solutions generated by the compact finite difference scheme with only the bound-preserving limiter are still highly oscillatory for
the Euler equation without the TVB limiter.

Notice that the oscillations in Figure \ref{incompressibleEulerVP1}  suggest that the artificial viscosity induced by the bound-preserving limiter is quite low.
 
  \begin{figure}[ht]
  \begin{subfigure}{.49\textwidth}
\includegraphics[scale=0.41]{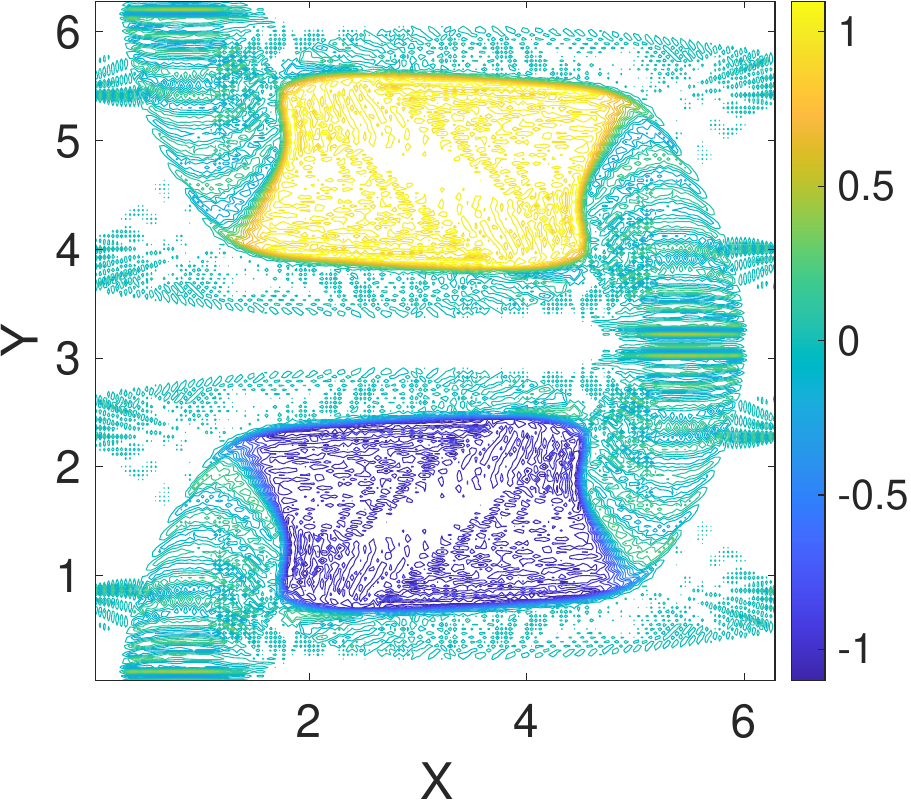}
\caption{$T=5$, with no limiter.}
\end{subfigure}
 \begin{subfigure}{.49\textwidth}
\includegraphics[scale=0.41]{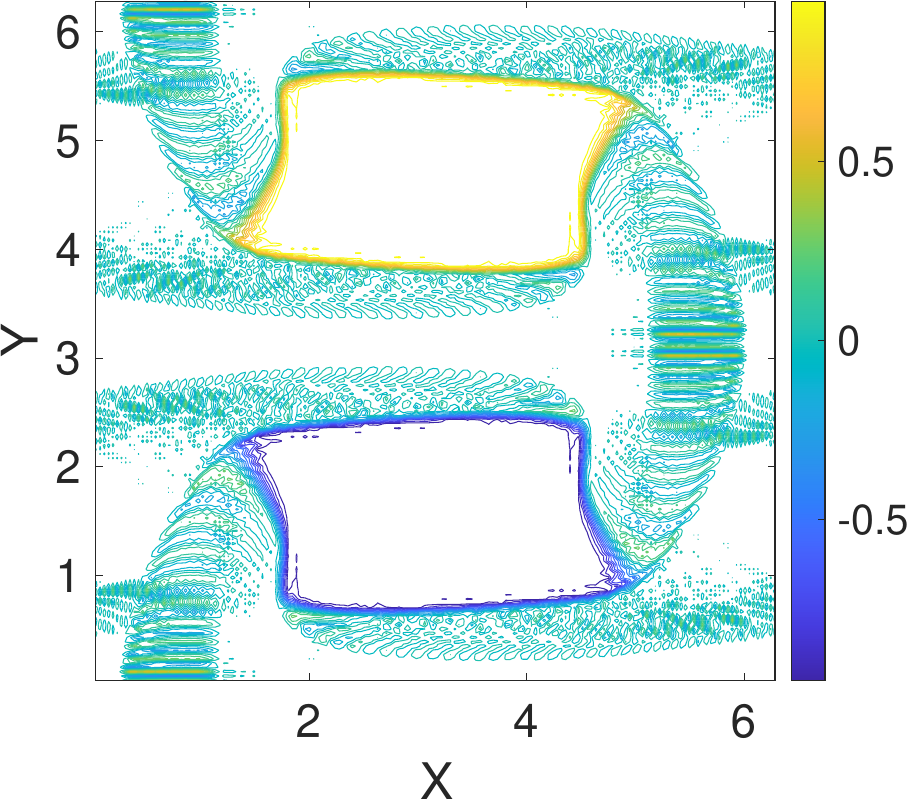}
\caption{$T=5$, with BP.}
\end{subfigure}

  \begin{subfigure}{.49\textwidth}
\includegraphics[scale=0.35]{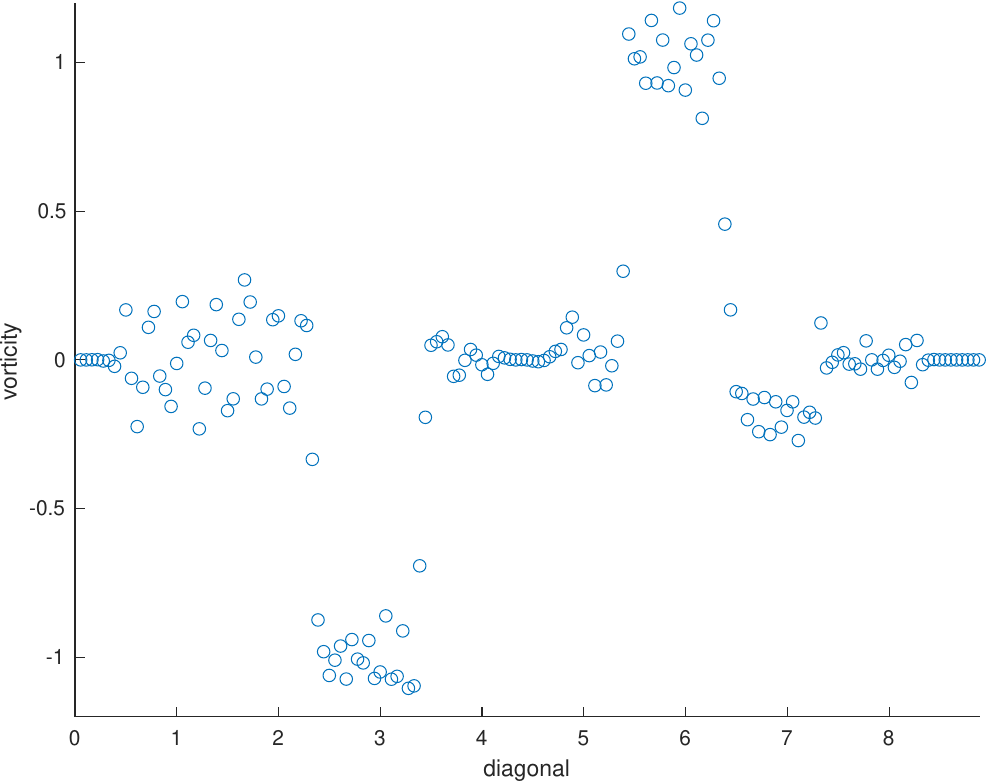}
\caption{$T=5$, with no limiter.}
\end{subfigure}
 \begin{subfigure}{.49\textwidth}
\includegraphics[scale=0.35]{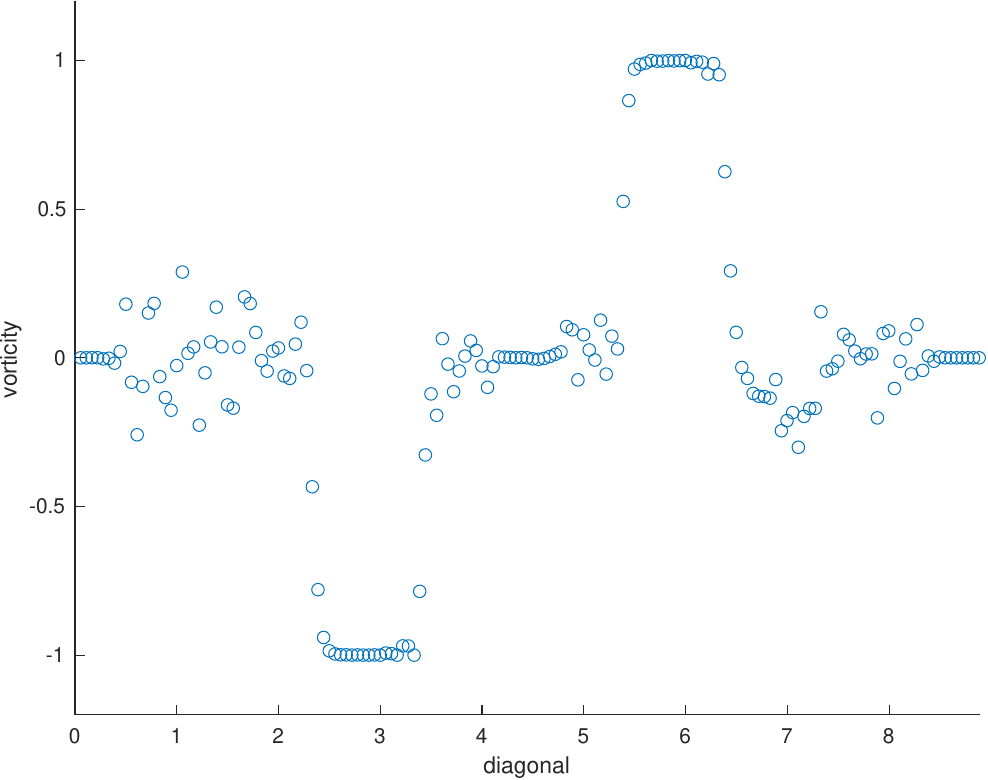}
\caption{$T=5$, with BP.}
\end{subfigure}
\caption{A fourth-order accurate compact finite difference scheme for the incompressible Euler equation at $T=5$ on a $160\times 160$ mesh. 
The time step is $\Delta t=\frac{1}{24\max|\mathbf u_0|}\Delta x$. The second row is the cut along the diagonal of the two-dimensional array. }
\label{incompressibleEulerVP1}
 \end{figure}

  \begin{figure}[ht]
\begin{subfigure}{.49\textwidth}
\includegraphics[scale=0.41]{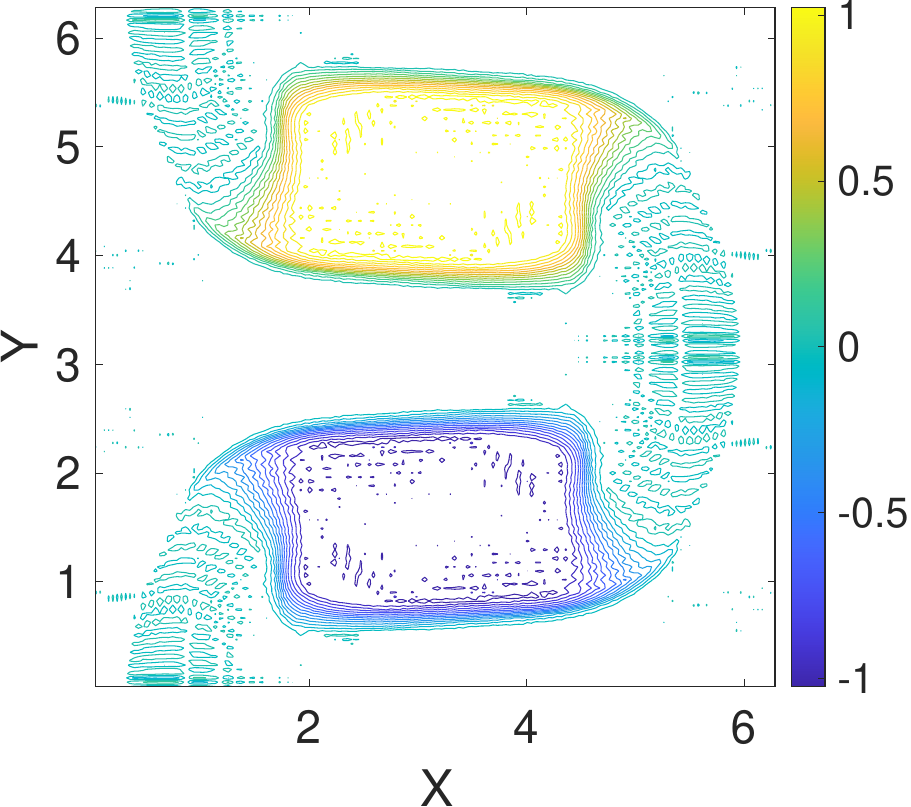}
\caption{$T=5$, with TVB1, P=10.}
\end{subfigure}
\begin{subfigure}{.49\textwidth}
\includegraphics[scale=0.41]{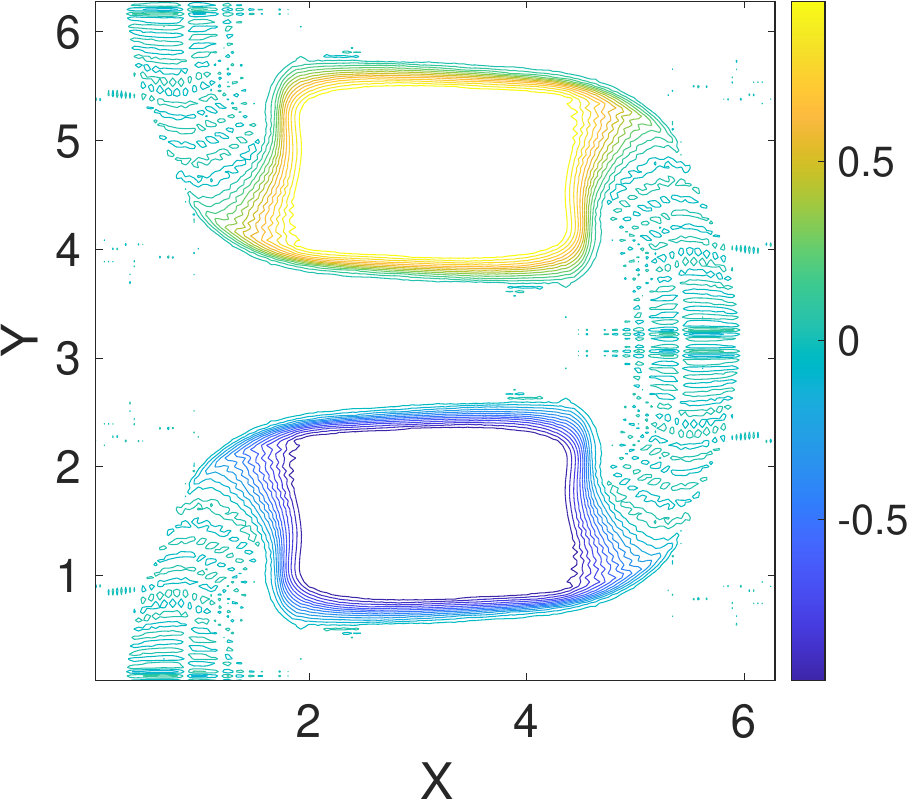}
\caption{$T=5$, with BP and TVB1, P=10.}
\end{subfigure}

\begin{subfigure}{.49\textwidth}
\includegraphics[scale=0.35]{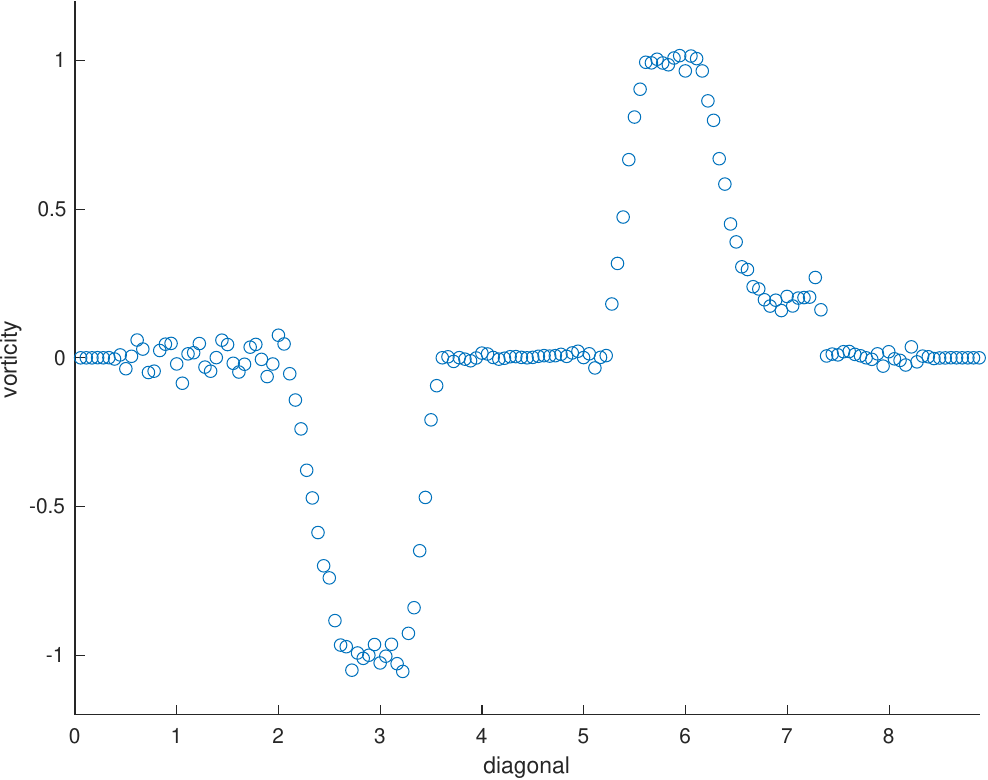}
\caption{$T=5$, with TVB1,  P = 10.}
\end{subfigure}
\begin{subfigure}{.49\textwidth}
\includegraphics[scale=0.35]{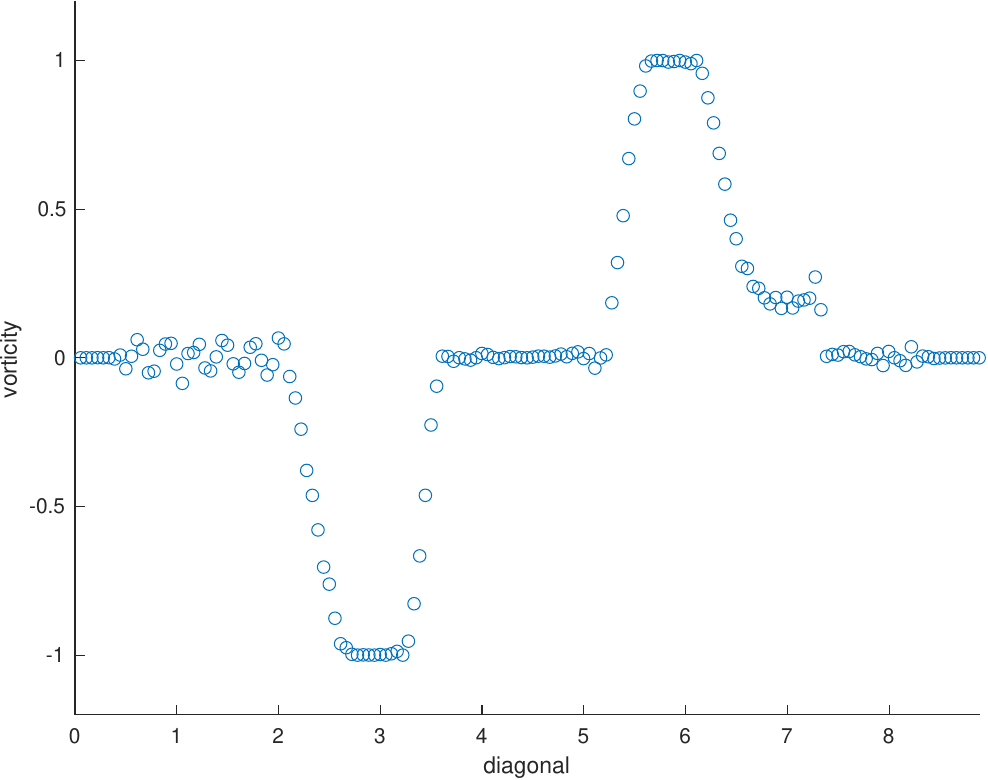}
\caption{$T=5$, with BP and TVB1, P = 10.}
\end{subfigure}
\caption{A fourth-order accurate compact finite difference scheme for the incompressible Euler equation at $T=5$ on a $160\times 160$ mesh. 
The time step is $\Delta t=\frac{1}{24\max|\mathbf u_0|}\Delta x$. The second row is the cut along the diagonal of the two-dimensional array. }
\label{incompressibleEulerVP2}
\end{figure}

\begin{figure}[ht]
\begin{subfigure}{.49\textwidth}
\includegraphics[scale=0.41]{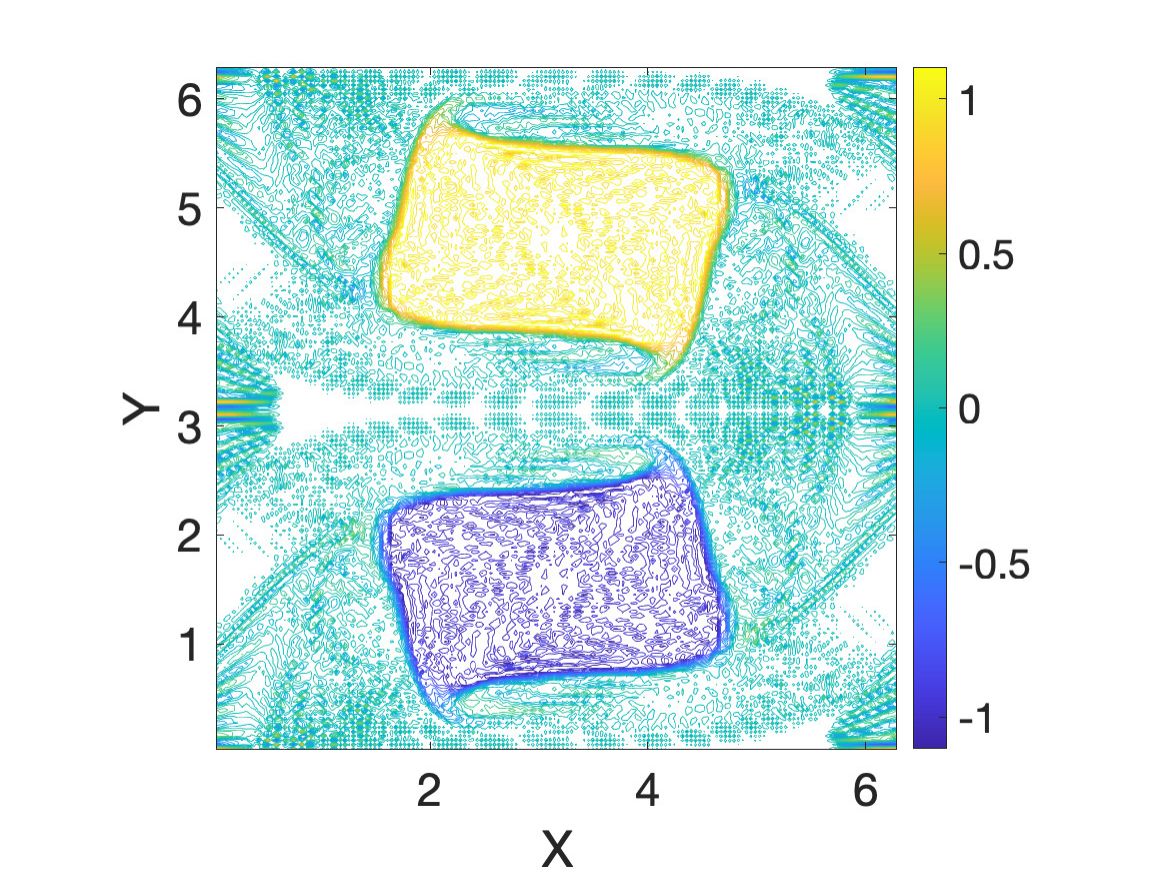}
\caption{$T=10$, with no limiter.}
\end{subfigure}
\begin{subfigure}{.49\textwidth}
\includegraphics[scale=0.41]{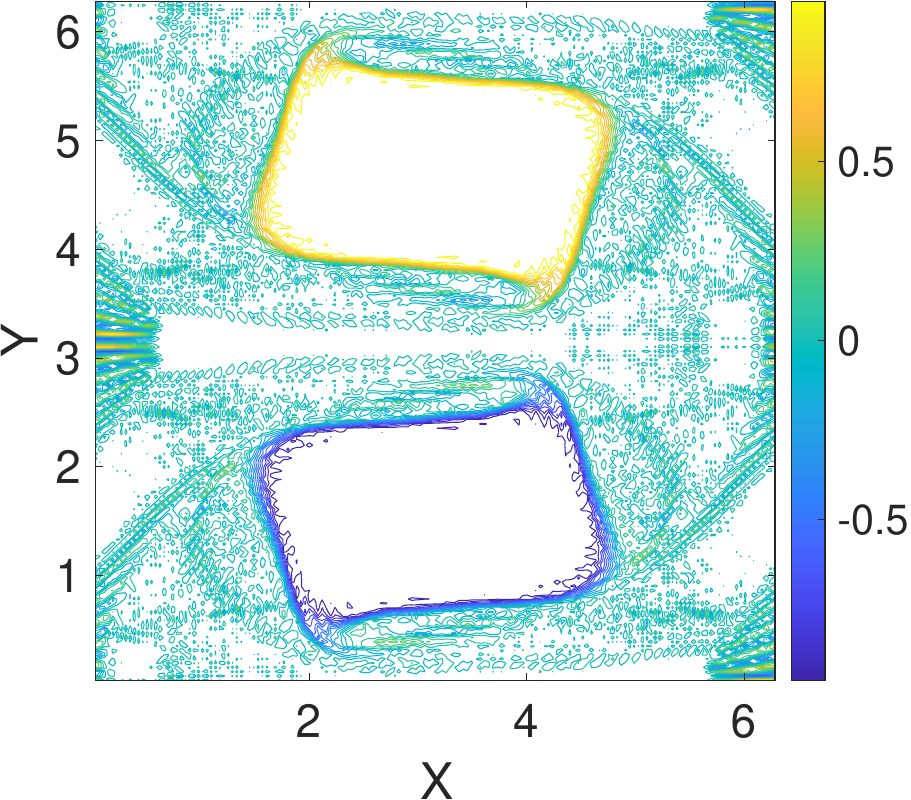}
\caption{$T=10$, with BP.}
\end{subfigure}

\begin{subfigure}{.49\textwidth}
\includegraphics[scale=0.35]{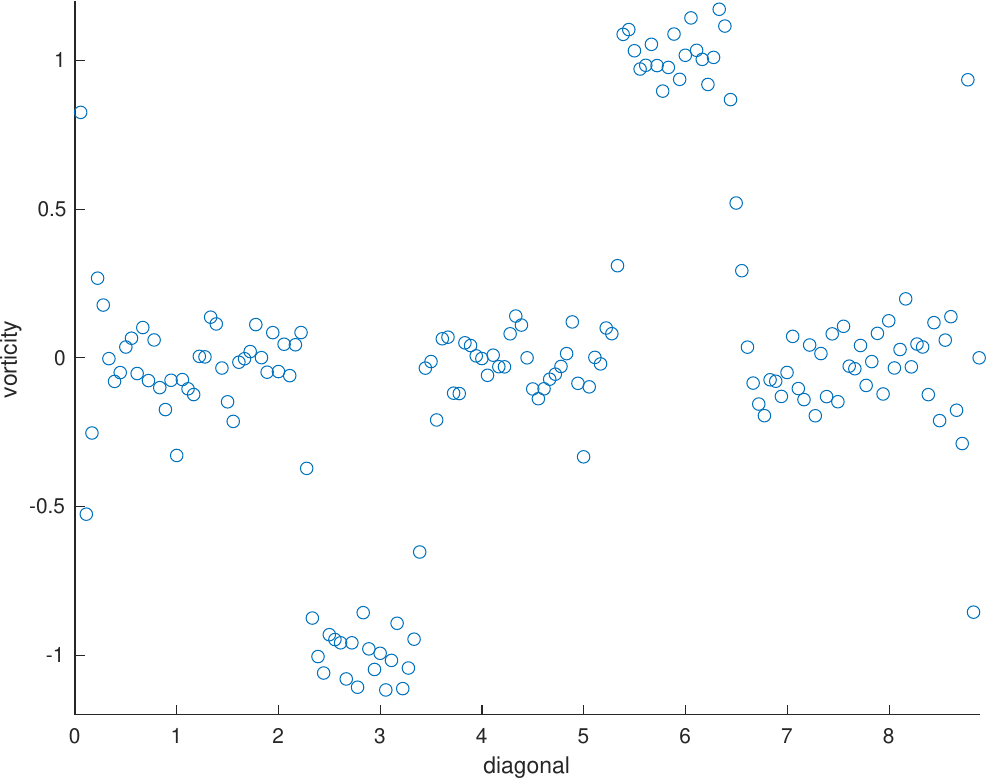}
\caption{$T=10$, with no limiter.}
\end{subfigure}
 \begin{subfigure}{.49\textwidth}
\includegraphics[scale=0.35]{VP_BP=1,TVB=0,T=5,N=160_diag-eps-converted-to.pdf}
\caption{$T=10$, with BP.}
\end{subfigure}
\caption{A fourth-order accurate compact finite difference scheme for the incompressible Euler equation at $T=5$ on a $160\times 160$ mesh. 
The time step is $\Delta t=\frac{1}{24\max|\mathbf u_0|}\Delta x$. The second row is the cut along the diagonal of the two-dimensional array. }
\label{incompressibleEulerVP3}
\end{figure}

\begin{figure}[ht]
\begin{subfigure}{.49\textwidth}
\includegraphics[scale=0.41]{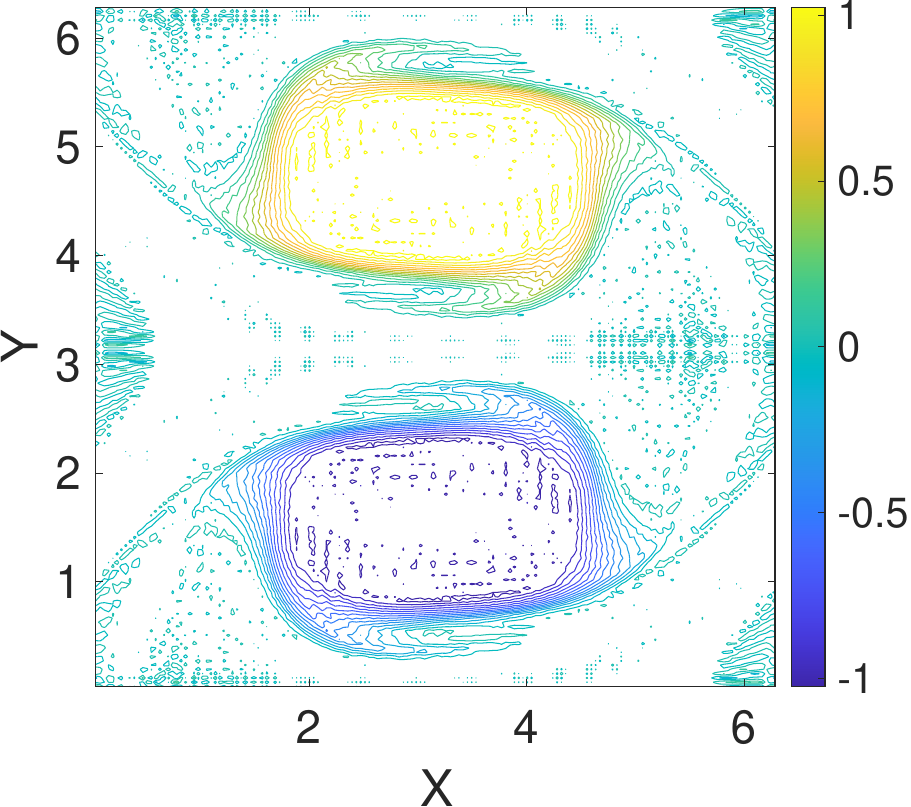}
\caption{$T=10$, with TVB1,  P=10.}
\end{subfigure}
\begin{subfigure}{.49\textwidth}
\includegraphics[scale=0.41]{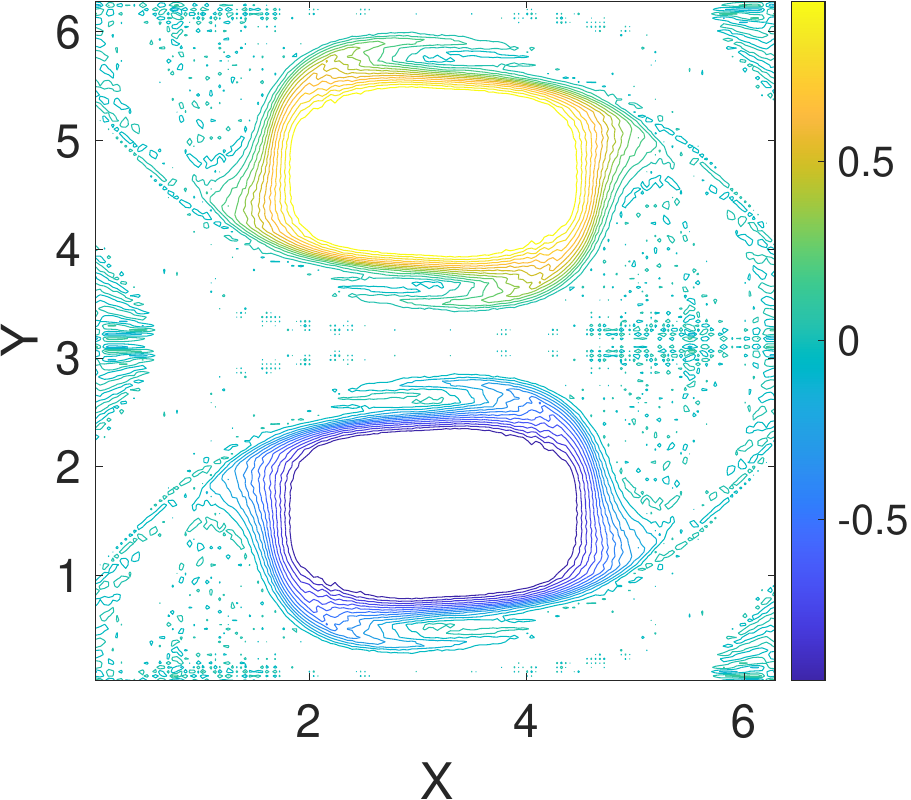}
\caption{$T=10$, with BP and TVB1, P=10.}
\end{subfigure}

\begin{subfigure}{.49\textwidth}
\includegraphics[scale=0.35]{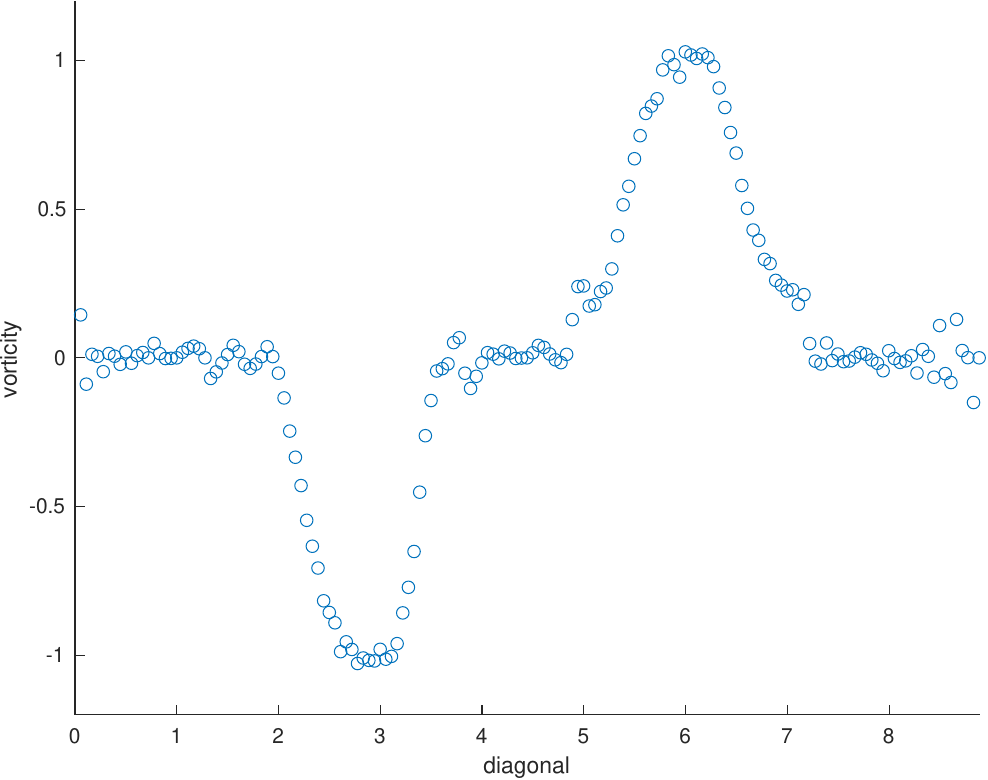}
\caption{$T=10$, with TVB1,  P = 10.}
\end{subfigure}
\begin{subfigure}{.49\textwidth}
\includegraphics[scale=0.35]{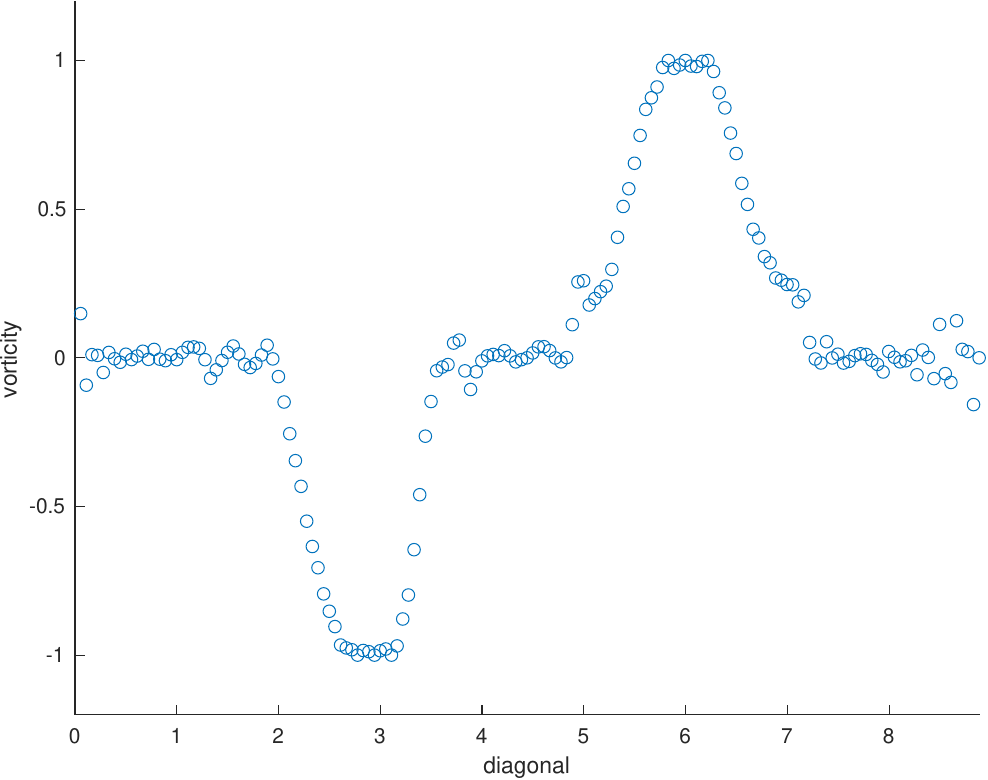}
\caption{$T=10$, with BP and TVB1, P = 10.}
\end{subfigure}
\caption{A fourth-order accurate compact finite difference scheme for the incompressible Euler equation at $T=10$ on a $160\times 160$ mesh. 
The time step is $\Delta t=\frac{1}{12\max|\mathbf u_0|}\Delta x$. The second row is the cut along the diagonal of the two-dimensional array.}
\label{incompressibleEulerVP4} 
 \end{figure}

 
\section{Concluding Remarks}
\label{sec-rmk}
We have proven that a simple limiter can preserve bounds for the fourth-order compact finite difference method solving 
 the two dimensional incompressible Euler equation,  with a discrete divergence-free velocity field.  We also prove that the TVB limiter modified from \cite{cockburn1994nonlinearly} does not affect the bound-preserving property of $\bar{\omega}$. With both the TVB limiter and the bound-preserving limiter, the numerical solutions of high order compact finite difference scheme can be rendered non-oscillatory and  strictly bound-preserving.
 
 For the sixth-order and eighth-order compact finite difference method for convection problem with weak monotonicty in \cite{li2018high},  the divergence-free velocity can be constructed accordingly, which gives a higher order bound-preserving scheme for the incompressible flow by applying Algorithm \ref{bp-limiter} for several times.  The TVB limiting procedure in Section \ref{sec-tvb-limiter-def} can also be defined with a similar result as Theorem \ref{tvb-avg-bp}.

\include{cFD_appendix.tex}

\section*{Statements and Declarations}

\subsection*{Funding} The research is supported by NSF DMS-1913120.
\subsection*{Conflict of interest}
On behalf of all authors, the corresponding author states that there is no conflict of interest. 

\bibliographystyle{siamplain}

\bibliography{references.bib}

\end{document}

%% file: cFD_appendix.tex
\section*{Appendix A: Comparison With The Nine-point Discrete Laplacian}
\setcounter{equation}{0}
\label{appendix1}

Consider solving the two-dimensional Poisson equations $u_{xx}+u_{yy}=f$ with either homogeneous Dirichlet boundary conditions or periodic boundary conditions on 
a rectangular domain.
Let $\mathbf u$ be a $N_x\times N_y$ matrix with entries $u_{i,j}$ denoting the numerical solutions at a uniform grid $(x_i, y_j)=(\frac{i}{Nx}, \frac{j}{Ny})$. 
Let $\mathbf f$ be a $N_x\times N_y$ matrix with entries $f_{i,j}=f(x_i, y_j)$.
The fourth order compact finite difference method in Section \ref{sec-cfd-review} for $u_{xx}+u_{yy}=f$ can be written as:
\begin{equation}
\frac{1}{\Delta x^2}W_{2x}^{-1}D_{xx} \mathbf u +\frac{1}{\Delta y^2}W_{2y}^{-1}D_{yy} \mathbf u=f(\mathbf{u}).
\label{poisson-scheme1} 
\end{equation}
For convenience, we introduce two matrices,
\[U=\begin{pmatrix*}[l]
  u_{i-1, j+1} &u_{i, j+1} &u_{i+1, j+1} \\
  u_{i-1, j} &u_{i, j} &u_{i+1, j} \\
  u_{i-1, j-1} &u_{i, j-1} &u_{i+1, j-1}
 \end{pmatrix*}, \quad F=\begin{pmatrix*}[l]
  f_{i-1, j+1} &f_{i, j+1} &f_{i+1, j+1} \\
  f_{i-1, j} &f_{i, j} &f_{i+1, j} \\
  f_{i-1, j-1} &f_{i, j-1} &f_{i+1, j-1}
 \end{pmatrix*}.\]
Notice that the scheme \eqref{poisson-scheme1} is equivalent to 
\[\frac{1}{\Delta x^2}W_{2y}D_{xx} \mathbf u +\frac{1}{\Delta y^2}W_{2x}D_{yy} \mathbf u=W_{2x}W_{2y}f(\mathbf{u}),\]
which can be written as
\begin{equation}
 \frac{1}{12\Delta x^2}\begin{pmatrix}
                       1 & -2 & 1\\
                       10 & -20 & 10\\
                       1 & -2 & 1 \end{pmatrix} : U
           + \frac{1}{12\Delta y^2}\begin{pmatrix}
                       1 & 10 & 1\\
                       -2 & -20 & -2\\
                       1 & 10 & 1 \end{pmatrix} : U  
  =
 \frac{1}{144}
 \begin{pmatrix}
                       1 & 10 & 1\\
                       10 & 100 & 10\\
                       1 & 10 & 1 \end{pmatrix}
                       :
  F,
\label{poisson-scheme2} 
\end{equation}
where $:$ denotes the sum of all entrywise products in two matrices of the same size.

In particular, if $\Delta x=\Delta y=h$, the scheme above reduces to
\[\frac{1}{6h^2} \begin{pmatrix}
                       1 & 4 & 1\\
                       4 & -20 & 4\\
                       1 & 4 & 1 \end{pmatrix}  
 :  U=
\frac{1}{144}
 \begin{pmatrix}
                       1 & 10 & 1\\
                       10 & 100 & 10\\
                       1 & 10 & 1 \end{pmatrix}:
 F. 
\]
Recall that the classical nine-point discrete Laplacian \cite{leveque2007finite} for the Poisson equation can be written as
\begin{equation}
\frac{1}{12\Delta x^2}\begin{pmatrix}
                       1 & -2 & 1\\
                       10 & -20 & 10\\
                       1 & -2 & 1 \end{pmatrix}: U
           + \frac{1}{12\Delta y^2}\begin{pmatrix}
                       1 & 10 & 1\\
                       -2 & -20 & -2\\
                       1 & 10 & 1 \end{pmatrix}: U  
   =
\frac{1}{12}
 \begin{pmatrix}
                       0 & 1 & 0\\
                       1 & 8 & 1\\
                       0 & 1 & 0 \end{pmatrix}:
 F, 
\label{poisson-scheme3} 
\end{equation}
which reduces to the following under the assumption $\Delta x=\Delta y=h$,
\[
\frac{1}{6h^2} \begin{pmatrix}
                       1 & 4 & 1\\
                       4 & -20 & 4\\
                       1 & 4 & 1 \end{pmatrix}  
:  U =
 \frac{1}{12}
 \begin{pmatrix}
                       0 & 1 & 0\\
                       1 & 8 & 1\\
                       0 & 1 & 0 \end{pmatrix}:
 F. 
\]
Both schemes \eqref{poisson-scheme2} and \eqref{poisson-scheme3} are fourth order accurate and they have the same stencil in the left hand side. As to 
which scheme produces smaller errors, it seems to be problem dependent,
 see Figure \ref{fig-poisson}. Figure \ref{fig-poisson} shows the errors of two schemes \eqref{poisson-scheme2} and \eqref{poisson-scheme3} using uniform grids with $\Delta x=\frac23 \Delta y$ for solving the Poisson equation $u_{xx}+u_{yy}=f$
 on a rectangle $[0,1]\times[0,2]$ with Dirichlet boundary conditions. For solution 1, we have $u(x,y)=\sin(\pi x)\sin(\pi y)+2x$, for solution 2, we have $u(x,y)=\sin(\pi x)\sin(\pi y)+4 x^4 y^4$.

\begin{figure}[htbp]
\begin{subfigure}{.5\textwidth}
\centering
\includegraphics[scale=0.2]{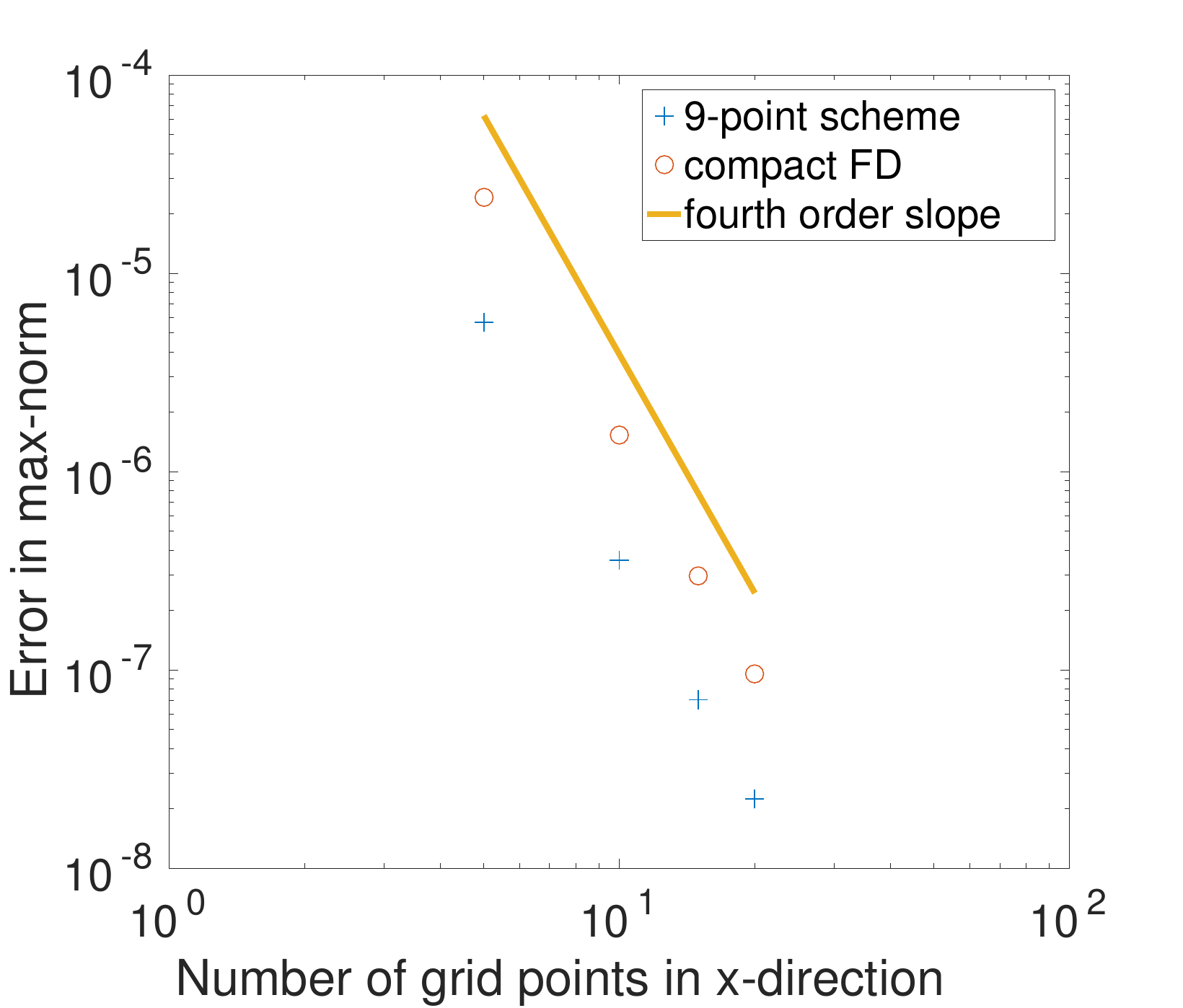}
\caption{Solution 1.}
\end{subfigure}
\begin{subfigure}{.49\textwidth}
\centering
\includegraphics[scale=0.2]{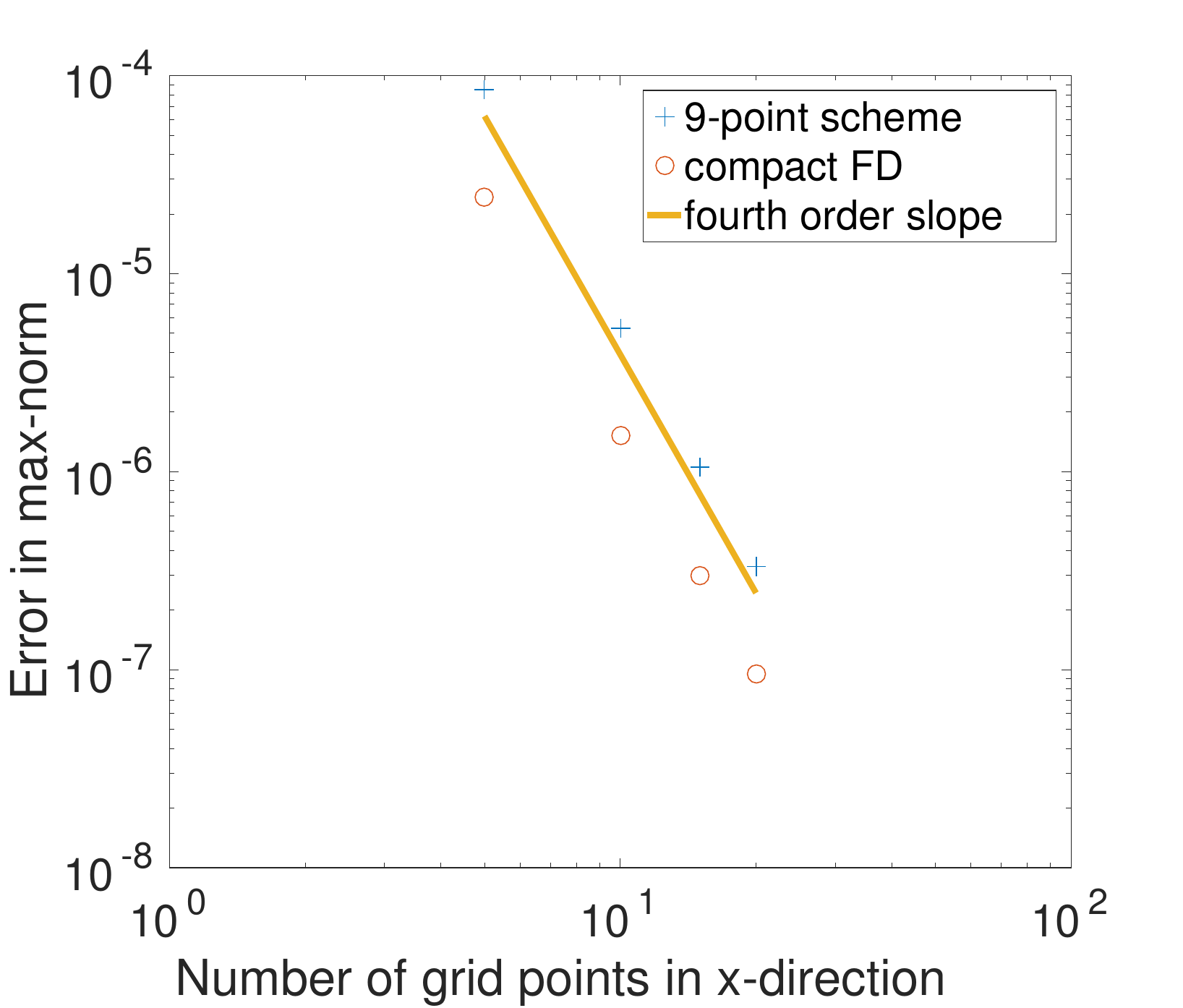}
\caption{Solution 2.}
\end{subfigure}
 \caption{Error comparison.}
 \label{fig-poisson}
\end{figure} 

\section*{Appendix B: $M$-matrices And A Discrete Maximum Principle}
\label{appendix2}
Consider solving the heat equation $u_t=u_{xx}+u_{yy}$ with a periodic boundary condition. It is well known that a discrete maximum principle is satisfied
under certain time step constraints if the spatial discretization is the nine-point discrete Laplacian or 
the compact scheme \eqref{poisson-scheme1} with backward Euler and Crank-Nicolson time discretizations.
For simplicity, we only consider the compact scheme \eqref{poisson-scheme1} and the discussion for the nine-point discrete Laplacian is similar. 
Assume $\Delta x=\Delta y=h$. For backward Euler, the scheme can be written as
\[
\frac{1}{144}
 \begin{pmatrix}
                       1 & 10 & 1\\
                       10 & 100 & 10\\
                       1 & 10 & 1 \end{pmatrix}  : ( U^{n+1}- U^{n})=\frac{\Delta t}{6h^2} \begin{pmatrix}
                       1 & 4 & 1\\
                       4 & -20 & 4\\
                       1 & 4 & 1 \end{pmatrix}  
:  U^{n+1} , 
 \]
thus
\[\frac{1}{144}
 \begin{pmatrix}
                       1 & 10 & 1\\
                       10 & 100 & 10\\
                       1 & 10 & 1 \end{pmatrix}:  U^{n+1}-\frac{\Delta t}{6h^2} \begin{pmatrix}
                       1 & 4 & 1\\
                       4 & -20 & 4\\
                       1 & 4 & 1 \end{pmatrix}:  U^{n+1}   =\frac{1}{144}
 \begin{pmatrix}
                       1 & 10 & 1\\
                       10 & 100 & 10\\
                       1 & 10 & 1 \end{pmatrix}
 :  U^n. \]
 Let $A$ and $B$ denote the matrices corresponding to the operator in the left hand side and right hand side above respectively, 
 then the scheme can be written as
 \[A \mathbf u^{n+1}= B \mathbf u^n,\]
and $A$ is a $M$-Matrix (diagonally dominant, positive diagonal entries and non-positive off diagonal entries) 
under the following constraint which allows very large time steps:
\[\frac{\Delta t}{h^2}\geq \frac{5}{48}.\]
The inverses of $M$-Matrices have non-negative entries, e.g., see \cite{qin2018implicit}. Thus $A^{-1}$ has non-negative entries. Moreover, it is straightforward to check that
$A\mathbf e=\mathbf e$ where $\mathbf e=\begin{pmatrix}
                                 1 & 1 & \cdots & 1
                                \end{pmatrix}^T.
$ Thus $A^{-1}\mathbf e=\mathbf e$, which implies the sum of each row of $A^{-1}$ is $1$ thus each row of $A^{-1}$ multiplying any vector $V$ is a convex 
combination of entries of $V$. It is also obvious that each entry of $B$ is non-negative and the sum of 
each row of $B$ is $1$. Therefore, $\mathbf u^{n+1}=A^{-1}B \mathbf u^n$ satisfies a discrete maximum principle:
\[ \min_{i,j} u^n_{i,j}\leq u^{n+1}_{i,j} \leq\max_{i,j} u^n_{i,j}.\]

For the second order accurate Crank-Nicolson time discretization, the scheme can be written as
\[ \frac{1}{144}
 \begin{pmatrix}
                       1 & 10 & 1\\
                       10 & 100 & 10\\
                       1 & 10 & 1 \end{pmatrix} : ( U^{n+1}- U^{n})=\frac{\Delta t}{6h^2} \begin{pmatrix}
                       1 & 4 & 1\\
                       4 & -20 & 4\\
                       1 & 4 & 1 \end{pmatrix}  
 : \frac{ U^{n+1}+ U^{n}}{2},  \]
 thus
\[
\left[\frac{1}{144}
 \begin{pmatrix}
                       1 & 10 & 1\\
                       10 & 100 & 10\\
                       1 & 10 & 1 \end{pmatrix}-\frac{\Delta t}{12h^2} \begin{pmatrix}
                       1 & 4 & 1\\
                       4 & -20 & 4\\
                       1 & 4 & 1 \end{pmatrix}   \right]:  U^{n+1}=\]\[
 \left[ \frac{1}{144}\begin{pmatrix}
                       1 & 10 & 1\\
                       10 & 100 & 10\\
                       1 & 10 & 1 \end{pmatrix}+\frac{\Delta t}{12h^2} \begin{pmatrix}
                       1 & 4 & 1\\
                       4 & -20 & 4\\
                       1 & 4 & 1 \end{pmatrix}  \right] 
 :  U^n.  \]
 Let the matrix-vector form of the scheme above be $A \mathbf u^{n+1}=B \mathbf u^n$. Then for $A$ to be an $M$-Matrix,  we only need $\frac{\Delta t}{h^2}\geq \frac{5}{24}$. However, for $B$ 
 to have non-negative entries, we need $\frac{\Delta t}{h^2}\leq \frac{5}{12}$. Thus the Crank-Nicolson method can ensure a discrete maximum principle if the time step satisfies,
 \[ \frac{5}{24}h^2\leq \Delta t \leq \frac{5}{12}h^2.\]

\section*{Appendix C: Fast Poisson Solvers}
\label{appendix3}
\subsection*{Dirichlet boundary conditions}
\label{appendix3-1}
Consider solving the Poisson equation $u_{xx}+u_{yy}=f(x,y)$ on a rectangular domain $[0,L_x]\times [0,L_y]$ with homogeneous Dirichlet boundary conditions.
Assume we use the grid $x_i=i \Delta x,$ $i=0,\cdots,N_x+1$ with uniform spacing $\Delta x=\frac{L_x}{N_x+1}$ for the $x$-variable 
and $y_j=j \Delta y,$ $j=0,\cdots,N_y+1$ with uniform spacing $\Delta y=\frac{L_y}{N_y+1}$ for $y$-variable.
Let $\mathbf u$ be a $N_x\times N_y$ matrix such that its $(i,j)$ entry $u_{i,j}$
is the numerical solution at interior grid points $(x_i, y_j)$.
Let $\mathbf F$ be a $(N_x+2)\times (N_y+2)$ matrix with entries $f(x_{i}, y_{j})$ for $i=0,\cdots,N_x+1$ and $j=0,\cdots,N_y+1$.

To obtain the matrix representation of the operator in \eqref{poisson-scheme2} and \eqref{poisson-scheme3}, we 
consider two operators:
\begin{itemize}
 \item Kronecker product of two matrices: if $A$ is $m \times n$ and $B$ is $p\times q$, then
 $A\otimes B$ is $mp\times nq$ give by
 \[A\otimes B=\begin{pmatrix}
                  a_{11}B & \cdots & a_{1n} B\\
                  \vdots & \vdots & \vdots\\
                  a_{m1}B & \cdots & a_{mn} B
                 \end{pmatrix}.
\]
\item For a $m\times n$ matrix $X$, $\vecm(X)$ denotes a column vector of size $mn$ made of the columns of $X$ stacked atop one another from left to right.
\end{itemize}

The following properties will be used:
\begin{enumerate}
 \item $(A \otimes B)(C \otimes D)=AC \otimes BD$.
 \item $(A \otimes B)^{-1}=A^{-1}\otimes B^{-1}$.
 \item $(B^T\otimes A)\vecm(X)=\vecm(AXB)$.
\end{enumerate}
We define two tridiagonal square matrices of size $N_x\times N_x$:
\[D_{xx}= \begin{pmatrix}
              -2 & 1 & & & &\\
              1 & -2 & 1 & & & \\              
              & 1 & -2 & 1 & & \\
              & & \ddots & \ddots & \ddots & \\
              & & & 1 & -2 &1 \\
              & & & & 1 & -2
             \end{pmatrix}, W_{2x}=\frac{1}{12} \begin{pmatrix}
              10& 1 & & & &\\
              1 & 10 & 1 & & & \\              
              & 1 & 10 & 1 & & \\
              & & \ddots & \ddots & \ddots & \\
              & & & 1 & 10 &1 \\
              & & & & 1 & 10
             \end{pmatrix}.\]
Let $\overline W_{2x}$ denote a $N_x\times (N_x+2)$ tridiagonal matrix of the following form:
\begin{equation}
\overline W_{2x}=\frac{1}{12} \begin{pmatrix}
              1& 10 & 1 & & &\\
               & 1 & 10 &1 & & \\              
               &  & \ddots & \ddots & \ddots\\
               & &  & 1 & 10 &1
             \end{pmatrix}.
             \label{appendix-barw2x}
\end{equation}
The matrices $D_{yy}$, $W_{2y}$ and $\overline W_{2y}$ are similarly defined. 

Then the scheme \eqref{poisson-scheme2} can be written in a matrix-vector form: 
\[ \frac{1}{\Delta x^2}D_{xx} \mathbf u W^T_{2y}+\frac{1}{\Delta y^2} W_{2x} \mathbf u D^T_{yy}= \overline W_{2x} \mathbf F\overline W^T_{2y},\]
or equivalently,
\begin{equation}
\left(W_{2y}\otimes \frac{1}{\Delta x^2} D_{xx}+\frac{1}{\Delta y^2} D_{yy}\otimes W_{2x}\right)\vecm( \mathbf u)=(\overline W_{2x}\otimes \overline W_{2y})\vecm( \mathbf F).
\label{poisson-matrix-1}
\end{equation}
Let $\mathbf{h}_x=[h_1,h_2,\cdots, h_{N_x}]^T$ with $h_i=\frac{i}{N_x+1}$, and $\sin(m\pi\mathbf h_x)$  denote 
a column vector of size $N_x$ with its $i$-th entry being $\sin(m\pi h_i)$. Then $\sin(m\pi \mathbf{h}_x)$ are the eigenvectors of $D_{xx}$ 
and $W_{2x}$ with the associated eigenvalues being $2\cos(\frac{m\pi}{N_x+1})-2$ and $\frac56+\frac16\cos(\frac{m\pi}{N_x+1})$ respectively for $m=1,\cdots, N_x$.
Let $$S_{x}=[\sin(\pi\mathbf h_x), \sin(2\pi\mathbf h_x),\cdots, \sin(N_x \pi\mathbf h_x)]$$ be the $N_x\times N_x$ eigenvector matrix, then 
$S_x$ is a symmetric matrix.
Let $\Lambda_{1x}$ denote a diagonal matrix with diagonal entries $2\cos(\frac{m\pi}{N_x+1})-2$
and  $\Lambda_{2x}$ denote a diagonal matrix with diagonal entries $\frac56+\frac16\cos(\frac{m\pi}{N_x+1})$, then 
we have $D_{xx}=S_x \Lambda_{1x}S_x^{-1}$ and $W_{2x}=S_x \Lambda_{2x}S_x^{-1}$, thus
\[  W_{2y} \otimes D_{xx}=(S_y \Lambda_{2y } S_y^{-1})\otimes( S_x \Lambda_{1x}S_x^{-1})=(S_y\otimes S_x)(\Lambda_{2y}\otimes \Lambda_{1x})(S_y^{-1}\otimes S_x^{-1}).\]
The scheme can be written as
\[(S_y\otimes S_x)(\frac{1}{\Delta x^2}\Lambda_{2y}\otimes \Lambda_{1x}+\frac{1}{\Delta y^2}\Lambda_{1y}\otimes \Lambda_{2x})(S_y^{-1}
\otimes S_x^{-1})\vecm(\mathbf u)=(\overline W_{2y}\otimes\overline  W_{2x})\vecm(\mathbf F).\]
Let $\Lambda$ be a 
$N_x\times N_y$ matrix with $\Lambda_{i,j}$ being equal to
\[\frac{1}{3\Delta x^2}\left(\cos(\frac{i\pi}{N_x+1})-1\right)
\left(\cos(\frac{m\pi}{N_y+1})+5\right)+\frac{1}{3 \Delta y^2}\left(\cos(\frac{m\pi}{N_x+1})+5\right) \left(\cos(\frac{j\pi}{N_y+1})-1\right), \]
then $\vecm(\Lambda)$ are precisely the diagonal entries of the diagonal matrix 
$\frac{1}{\Delta x^2}\Lambda_{2y}\otimes \Lambda_{1x}+\frac{1}{\Delta y^2}\Lambda_{1y}\otimes \Lambda_{2x}$, 
then the scheme above is equivalent to 
\[  S_x (\Lambda\circ(S_x^{-1}\mathbf u S_y^{-1})) S_y= \overline W_{2x}\mathbf  F\overline W^T_{2y},\]
where the symmetry of  $S$ matrices is used.
The solution is given by
\begin{equation}
  \mathbf u = S_x\{[S_x^{-1}(\overline W_{2x}\mathbf F\overline W^T_{2y})S_y^{-1}]./\Lambda \}S_y,
  \label{poisson-sol1}
\end{equation}
where $./$ denotes the entrywise division for two matrices of the same size.

Since the multiplication of the matrices $S$ and $S^{-1}$ can be implemented by the {\it Discrete Sine Transform},  \eqref{poisson-sol1} gives a fast Poisson solver.

For nonhomogeneous Dirichlet boundary conditions, the fourth order accurate compact finite difference scheme can also be written in the form of \eqref{poisson-matrix-1}:
\begin{equation}
\left(W_{2y}\otimes \frac{1}{\Delta x^2} D_{xx}+\frac{1}{\Delta y^2} D_{yy}\otimes W_{2x}\right)\vecm(\mathbf u)=\vecm(\tilde {\mathbf F}),
\label{poisson-matrix-2}
\end{equation}
where $\tilde{\mathbf F}$ consists of both $\mathbf F$ and the Dirichlet boundary conditions. Thus the scheme can still be efficiently implemented by
the {\it Discrete Sine Transform}.

\subsection*{Periodic boundary conditions}
 For periodic boundary conditions on a rectangular domain, we should consider the uniform grid 
 $x_i=i \Delta x,$ $i=1,\cdots,N_x$ with $\Delta x=\frac{L_x}{N_x}$
and $y_j=j \Delta y,$ $j=1,\cdots,N_y$ with uniform spacing $\Delta y=\frac{L_y}{N_y}$, then the 
fourth order accurate compact finite difference scheme can still be written in the form of \eqref{poisson-matrix-1}
with the $D_{xx}$, $D_{yy}$, $W_{2x}$ and $W_{2y}$ matrices being redefined as circulant matrices:
\[D_{xx}= \begin{pmatrix}
              -2 & 1 & & & &1\\
              1 & -2 & 1 & & & \\              
              & 1 & -2 & 1 & & \\
              & & \ddots & \ddots & \ddots & \\
              & & & 1 & -2 &1 \\
              1 & & & & 1 & -2
             \end{pmatrix}, W_{2x}=\frac{1}{12} \begin{pmatrix}
              10& 1 & & & & 1\\
              1 & 10 & 1 & & & \\              
              & 1 & 10 & 1 & & \\
              & & \ddots & \ddots & \ddots & \\
              & & & 1 & 10 &1 \\
              1& & & & 1 & 10
             \end{pmatrix}.\]
The Discrete Fourier Matrix is the eigenvector matrix for any circulant matrices, and the corresponding eigenvalues
are for $D_{xx}$ and $W_{2x}$ are $2\cos(\frac{m2\pi}{N_x})-2$ and $\frac16\cos(\frac{m2\pi}{N_x})+\frac56$ for $m=0,1,2,\cdots, Nx-1$. 
The  matrix $W_{2y}\otimes \frac{1}{\Delta x^2} D_{xx}+\frac{1}{\Delta y^2} D_{yy}\otimes W_{2x}$
is singular because its first eigenvalue $\Lambda_{1,1}$ is zero. Nonetheless, the scheme can still be implemented by solving \eqref{poisson-sol1} with Fast Fourier Transform. For the zero eigenvalue, 
we can simply reset the division by eigenvalue zero to zero. Since the eigenvector for eigenvalue zero is $\mathbf e=\begin{pmatrix}
                                 1 & 1 & \cdots & 1
                                \end{pmatrix}^T$, and 
the columns of the Discrete Fourier Matrix are orthogonal to one another,
resetting the division by eigenvalue zero to zero simply means that we obtain a numerical solution satisfying $\sum_i \sum_j u_{i,j}=0$.
And this is also the least square solution to the singular linear system.

\subsection*{Neumann boundary conditions}
For Dirichlet and periodic boundary conditions, we can invert the matrix coefficient matrix in   \eqref{poisson-matrix-1}
using eigenvectors of much smaller matrices $W_{2x}$ and $D_{xx}$ due to the fact that $W_{2x}-\frac{1}{12}D_{xx}$ is the identity matrix $Id$. 
Here we discuss how to achieve a fourth order accurate boundary approximation for Neumann boundary conditions by keeping $W_{2x}-\frac{1}{12}D_{xx}=Id$. 
We first consider a one-dimensional problem with homogeneous Neumann boundary conditions:
\begin{align*}
 &u''(x)=f(x), x\in[0,L_x], \\ 
 & u'(0)=u'(L_x)=0.   
  \end{align*}
Assume we use the uniform 
grid $x_i=i \Delta x,$ $i=0,\cdots,N_x+1$ with $\Delta x=\frac{L_x}{N_x+1}$. The two boundary point values
$u_0$ and $u_{N_x+1}$ can be expressed in terms of interior point values through boundary conditions.
For approximating the boundary conditions, we can apply the fourth order one-sided difference at $x=0$:
\[ u'(0)\approx\frac{-25 u(0)+48 u(\Delta x) -36 u(2\Delta x)+16u(3\Delta x)-3u(4\Delta x) }{12 \Delta x}\]
which implies the finite difference approximation: 
\[u_{0}=\frac{48 u_1-36 u_{2} +16 u_3-3 u_4}{25}.\]
Define two column vectors:
\[\mathbf u=[u_1 , u_2, \cdots, u_{N_x}]^T,\quad \mathbf f=[f(x_0), f(x_1),\cdots, f(x_{N_x}), f(x_{N_x+1}) ]^T,\]
then a fourth order accurate compact finite difference scheme can be written as 
\[ \frac{1}{\Delta x^2} \overline D_{xx} I_{x}\mathbf u=\overline W_{2x} \mathbf f,\]
where $\overline W_{2x}$ is the same as in \eqref{appendix-barw2x}, and $\overline D_{xx}$ is a matrix of size $N_x\times(N_x+2)$ and $I_{x}$
is a matrix of size $(N_x+2)\times N_x$:
\[\overline D_{xx}= \begin{pmatrix}
              1 & -2 & 1 & & & \\              
              & 1 & -2 & 1 & & \\
              & & \ddots & \ddots & \ddots & \\
              & & & 1 & -2 &1 
             \end{pmatrix}, I_{x}= \begin{pmatrix}
              \frac{48}{25} & -\frac{36}{25} &\frac{16}{25} &-\frac{3}{25} & &\\
               & 1 &  & & & \\              
              &  & 1 &  & & \\
              & &  & \ddots &  & \\
              & & &  & 1 & \\
               & & -\frac{3}{25} & \frac{16}{25} & -\frac{36}{25} &       \frac{48}{25}
             \end{pmatrix}.\]
Now consider solving the Poisson equation $u_{xx}+u_{yy}=f(x,y)$ on a rectangular domain $[0,L_x]\times [0,L_y]$ with homogeneous Neumann boundary conditions.
Assume we use the grid  $x_i=i \Delta x,$ $i=0,\cdots,N_x+1$ with $\Delta x=\frac{L_x}{N_x+1}$
and $y_j=j \Delta y,$ $j=0,\cdots,N_y+1$ with uniform spacing $\Delta y=\frac{L_y}{N_y+1}$.
Let $\mathbf u$ be a $N_x\times N_y$ matrix such that $u_{i,j}$ is the numerical solution at $(x_i, y_j)$ and
$\mathbf F$ be a $(N_x+2)\times (N_y+2)$ matrix with entries $f(x_{i}, y_{j})$ ($i=0,\cdots, N_x+1$, $j=0,\cdots, N_y+1$). 
Then a fourth order accurate compact finite difference scheme can be written as
\[ \frac{1}{\Delta x^2}\overline D_{xx}I_x \mathbf u I_y^T \overline{W}_{2y}^T+\frac{1}{\Delta y^2}\overline W_{2x}I_x \mathbf u I_y^T \overline{D}_{yy}^T=  \overline W_{2x}\mathbf F\overline W^T_{2y}.\]
Let $D_{xx}=\overline D_{xx}I_x$ and $W_{2x}=\overline W_{2x}I_x$, then the scheme can be written as \eqref{poisson-matrix-1}.

Notice that 
$W_{2x}-\frac{1}{12}D_{xx}=(\overline W_{2x}-\frac{1}{12}\overline D_{xx})I_x$ is still the identity matrix thus $W_{2x}$ and $D_{xx}$ still have the same eigenvectors.
Let $S$ be the eigenvector matrix and $\Lambda_1$ and $\Lambda_2$
be diagonal matrices with eigenvalues, then the scheme can still be implemented as \eqref{poisson-sol1}. The eigenvectors $S$ and the eigenvalues can be obtained 
by computing eigenvalue problems for two small matrices $D_{xx}$ of size $N_x\times N_x$ and $D_{yy}$ of size $N_y\times N_y$. If such a
Poisson problem needs to be solved in each time step in a time-dependent problem such as the incompressible flow equations, then 
this is an efficient Poisson solver because $S$ and $\Lambda$ can be computed 
before time evolution without considering eigenvalue problems for any matrix of size $N_xN_y\times N_xN_y$. 

For nonhomogeneous Neumann boundary conditions, 
the point values of $u$ along the boundary should be expressed in terms of interior ones as follows:
\begin{enumerate}
 \item First obtain the point values except the two cell ends (i.e., corner points of the rectangular domain) for each of the four boundary line segments. For instance,
 if the left boundary condition is $\frac{\partial u}{\partial x}(0,y)=g(y)$, then we obtain 
 \[u_{0,j}=\frac{48 u_{1,j}-36 u_{2,j} +16 u_{3,j}-3 u_{4,j}+12\Delta x g(y_j)}{25},\quad j=1,\cdots, N_y.\]
 \item Second, obtain the approximation at four corners using the point values along the boundary. For instance, if the bottom boundary condition is $\frac{\partial u}{\partial y}(x,0)=h(x)$, then 
 \[u_{0,0}=\frac{48 u_{1,0}-36 u_{2,0} +16 u_{3,0}-3 u_{4,0}+12\Delta y h(0)}{25}\]
\end{enumerate}
The scheme can still be written as \eqref{poisson-matrix-2} with $\tilde{\mathbf F}$ consisting of ${\mathbf F}$ and the nonhomogeneous boundary conditions. 
Notice that the matrix in \eqref{poisson-matrix-2} is singular thus we need to reset the division by eigenvalue zero to zero, which  however no longer means that the obtained
solution satisfies $\sum_i\sum_j u_{i,j}=0$ since the eigenvectors are not necessarily orthogonal to one another. 
See Figure \ref{fig-poisson-neumann} for the accuracy test of the fourth order compact finite difference scheme using uniform grids with $\Delta x=\frac32 \Delta y$ 
 for solving the Poisson equation $u_{xx}+u_{yy}=f$
 on a rectangle $[0,1]\times[0,2]$ with nonhomogeneous Neumann boundary conditions. The exact solution is $u(x,y)=\cos(\pi x)\cos(3\pi y)+\sin(\pi y)+x^4$. Since the solutions to Neumann boundary conditions are unique up to any constant, when computing
errors, we need to add a constant $\frac{1}{N_x}\frac{1}{N_y}\sum_{i,j} [u(x_i, y_j)-u_{i,j}]$ to each entry of $\mathbf u$.

\begin{figure}[htbp]
\begin{subfigure}{.5\textwidth}
\centering
\includegraphics[scale=0.2]{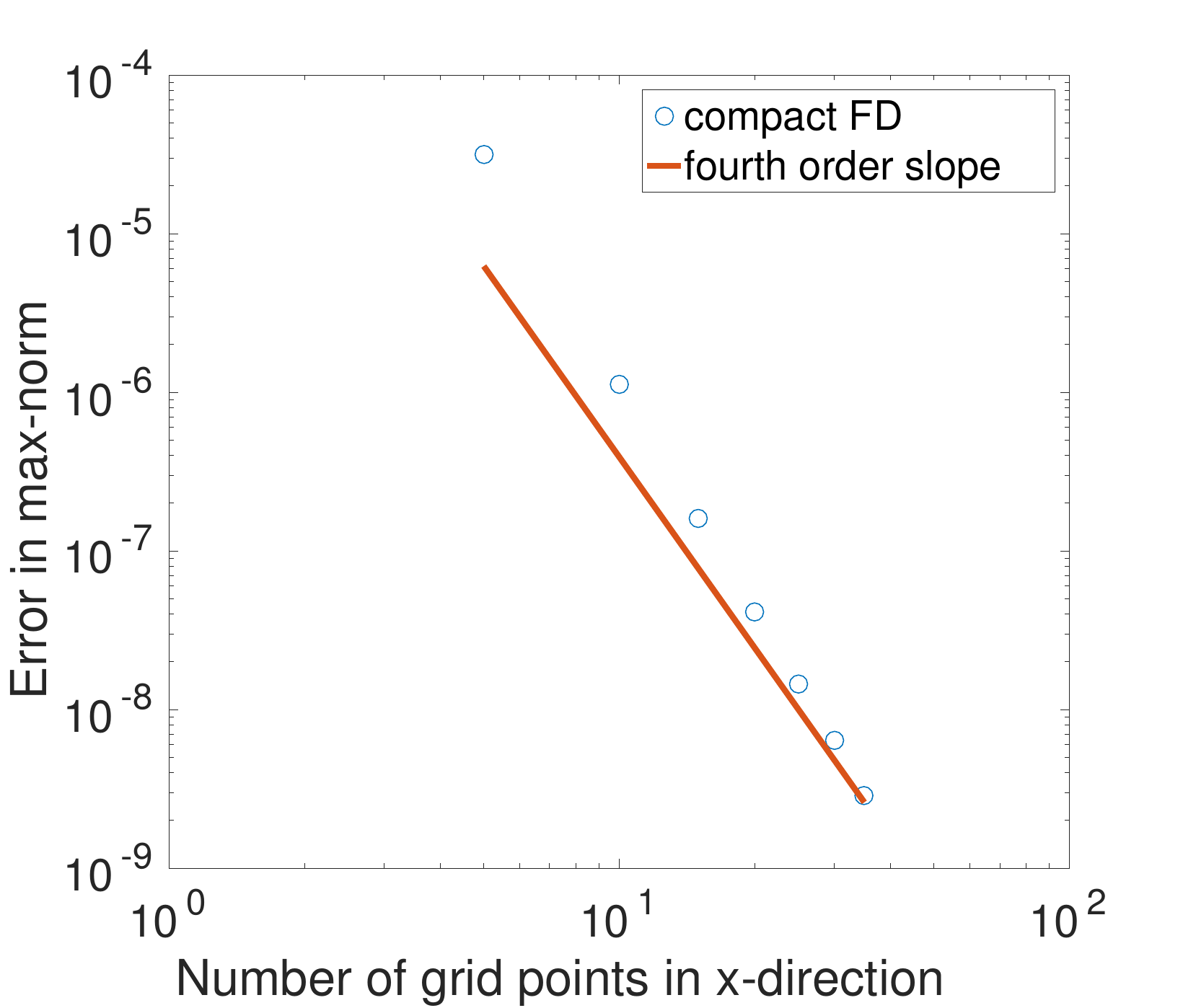}
\caption{Convergence rate.}
\end{subfigure}
\begin{subfigure}{.49\textwidth}
\centering
\includegraphics[scale=0.2]{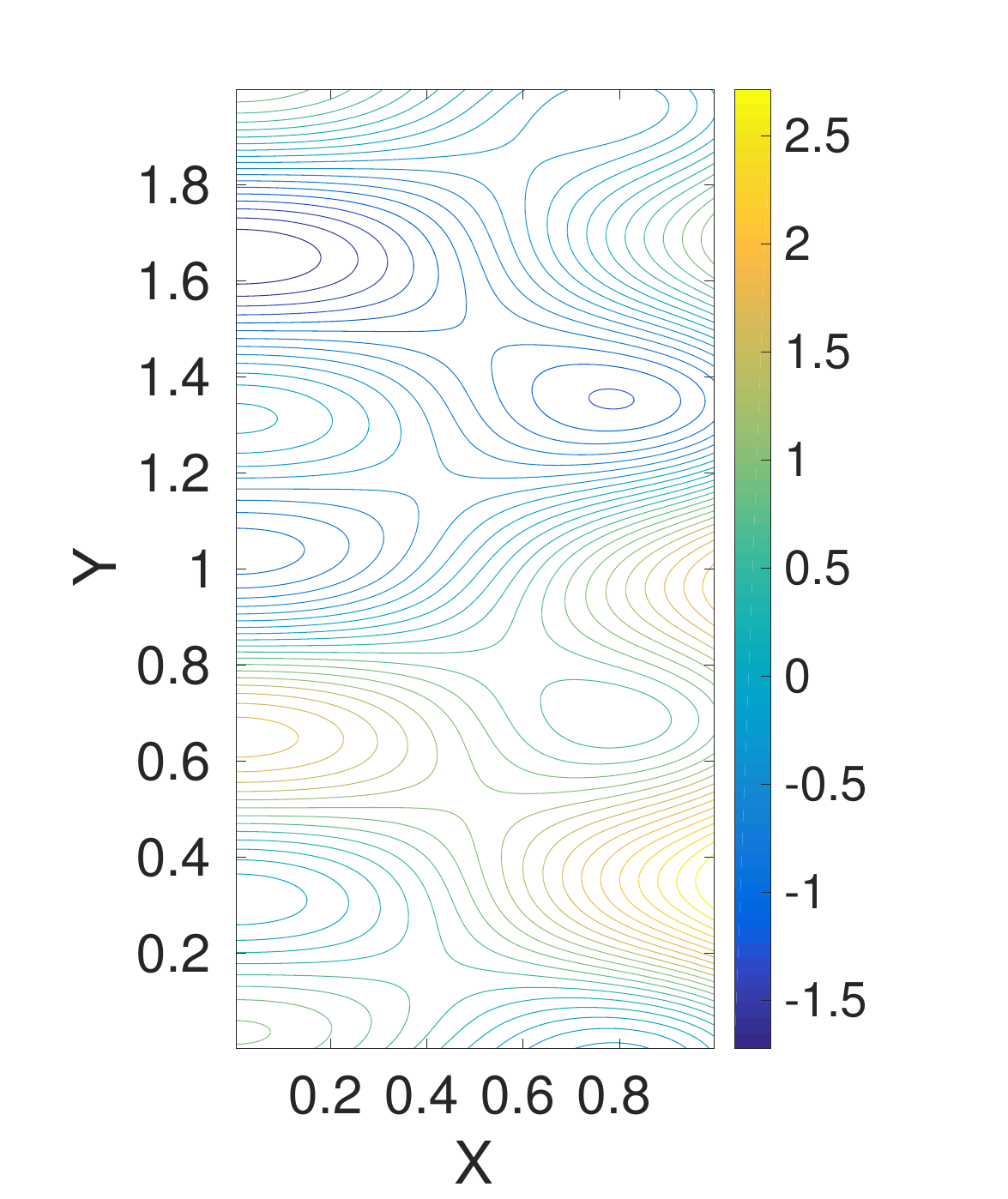}
\caption{The contour of the solution.}
\end{subfigure}
 \caption{Accuracy test for Neumann boundary condition.}
 \label{fig-poisson-neumann}
\end{figure} 